\documentclass[aap, press]{apt} 

\renewcommand{\crnotice}{}
\usepackage[colorlinks = true,linkcolor = blue,urlcolor  = blue,citecolor = blue,anchorcolor = blue]{hyperref}
\usepackage{times}
\usepackage{bbm}
\usepackage{stmaryrd}
\usepackage{enumitem}
\usepackage{natbib}

\setcitestyle{numbers}
\DeclareMathOperator{\Fix}{Fix}
\DeclareMathOperator{\Id}{Id}
\DeclareMathOperator{\Cov}{Cov}
\DeclareMathOperator{\Var}{Var}
\DeclareMathOperator{\diam}{diam}

\DeclareMathOperator{\Tr}{Tr}
\DeclareMathOperator{\Aut}{Aut}
\DeclareMathOperator{\Neg}{Neg}
\newcommand{\ds}{\displaystyle}
\newcommand{\eps}{\varepsilon}
\newcommand{\bigma}{\boldsymbol{\sigma}}
\newcommand{\bi}{\boldsymbol{\pi}}
\newcommand{\llb}{\llbracket}
\newcommand{\rrb}{\rrbracket}
\newcommand{\phe}{\varphi}
\newcommand{\prt}{\mathcal{Z}}
\newcommand{\dsum}{\ds\sum}
\newcommand{\1}{\mathbbm{1}}

\newcommand{\E}{\mathbb{E}}

\newcommand{\R}{\mathbb{R}}
\newcommand{\mc}{\mathcal}
\newcommand{\Proba}{\mathbb{P}}
\newcommand{\toPr}{\underset{\Proba}{\longrightarrow}}
\newcommand{\defeq}{\underset{\text{def}}{=}}
\newcommand{\qed}{\hfill $\Box$}
\newtheorem{conjecture}{Conjecture}

\authornames{Louis Vassaux, Laurent Massoulié} 
\shorttitle{The feasibility of multi-graph alignment} 

\begin{document}

\title{The feasibility of multi-graph alignment: a Bayesian approach} 

\authorone[INRIA, DMA/ENS]{Louis Vassaux} 
\authortwo[INRIA, DI/ENS]{Laurent Massoulié}

\address{INRIA Paris, PSL Research University, Paris, France} 
\emailone{louis.vassaux@inria.fr} 

\emailtwo{laurent.massoulie@inria.fr}

\begin{abstract}
We establish thresholds for the feasibility of random multi-graph alignment in two models. In the Gaussian model, we demonstrate an "all-or-nothing" phenomenon: above a critical threshold, exact alignment is achievable with high probability, while below it, even partial alignment is statistically impossible. In the sparse Erdős–Rényi model, we rigorously identify a threshold below which no meaningful partial alignment is possible and conjecture that above this threshold, partial alignment can be achieved. To prove these results, we develop a general Bayesian estimation framework over metric spaces, which provides insight into a broader class of high-dimensional statistical problems.
\end{abstract}



\section{Introduction}

\subsection{The graph alignment problem}

There are various statistical problems where we observe a network of linked objects which can be represented by a large graph; for instance, a network of protein-protein interactions or a social network. The \textbf{graph alignment problem}, in a broad sense, is the following: we observe the same underlying data multiple times, obtaining $p \geq 2$ \textbf{unlabelled} graphs which should be very similar, and attempt to retrieve the underlying isomorphism which pairs up any two nodes representing the same element. 

Formally, the problem can be described as follows. Assume that we have an underlying $p$-tuple of labelled random graphs $(G^{(1)}, \ldots, G^{(p)})$ with vertex set $\llb 1,n \rrb$. We select $\pi^* = (\pi_1^* = \Id, \pi_2^*, \ldots, \pi_p^*)$ with $\pi_2^* , \ldots, \pi_p^* $ chosen i.i.d,  uniformly at random, over $\mathcal{S}_n$; we observe 

\begin{equation} \label{defipi}
    (\bi^*)^{-1}(G) \defeq (G_{(\pi_1^*)^{-1}(e)}^{(1)})_{e \in E}, \ldots, (G_{(\pi_p^*)^{-1}(e)}^{(p)})_{e \in E}
\end{equation}

where $E$ is the set of edges in the complete graph on $\llb 1, n \rrb$, and $\mathcal{S}_n$ acts on $E$ via $\sigma(\{ u, v \}) \defeq \{ \sigma(u), \sigma(v) \}$. Our aim is to recover $\pi_2^*, \ldots, \pi_p^*$.

Typically, in order to analyse the problem for "typical" values of $G$, we assume that $G$ is also a random variable, such that the $(G_e^{(i)})_{1 \leq i \leq p}$ are correlated for any fixed $e \in E$. Two important examples are the following models:
\begin{itemize}
    \item the \textbf{Gaussian model} $\mathcal{N}(n, \rho)$, where $G^{(1)}, \ldots, G^{(p)}$ are weighted graphs: the weights $(G_e^{(i)})_{1 \leq i \leq p, e \in E}$ are drawn as a centred normal vector over $\R^E \otimes \R^p$, with the following covariance matrix:

    \begin{equation} \label{covdef}
        \Cov(G_{e}^{(i)}, G_{e'}^{(j)}) = 
        \begin{cases}
            1 & \text{ if } e=e', i=j \\
            \rho \in [0, 1] & \text{ if } e=e', i \neq j \\
            0 & \text{ otherwise}
            
        \end{cases}
\end{equation}
\item the \textbf{(correlated) Erdős–Rényi model} $\mathcal{G}(n, \lambda, s)$, in which $G^{(1)}, \ldots, G^{(p)}$ are unweighted graphs: we sample a "master graph" $G^{(0)}$ from the standard Erdős–Rényi model $\mathcal{G}(n, \frac{\lambda}{n})$ (for some $\lambda > 0$); then, the $(G^{(i)})$ are subgraphs of $G^{(0)}$, which are mutually independent conditional on $G^{(0)}$, and which contain any given edge of $G^{(0)}$ with a (fixed) probability $0 < s < 1$.
\end{itemize}

\subsection{Previous works}
The graph alignment problem has been studied extensively in the $2$-graph case, particularly for the two toy models described above. Much attention has been given to finding algorithms which, with high probability, recover the underlying permutation in reasonable time (\cite{relax},  \cite{feizi2017}, \cite{degrees}, \cite{spectral1}, \cite{spectral2}, \cite{spectral3}, \cite{chandeliers}, \cite{eralgo}, \cite{geom2}). These algorithms, in turn, can be used in various applications: for instance, de-anonymisation problems (\cite{deanom}), natural language processing (\cite{nlp}), or the analysis of protein interaction networks (\cite{protein}).

Another line of research consists in the information-theoretical analysis of the alignment problem, i.e. understanding when alignment is feasible (without regards to time complexity). Again, much progress has been made on the subject (\cite{exact}, \cite{almostexact}, \cite{partialimposs}, \cite{threshold}, \cite{lucagauss}, \cite{settling}) and we have a good understanding of the cutoffs for alignment feasibility. Note that, often enough, there is a real difference between feasibility and polynomial-time feasibility: there can be a "hard phase" where alignment is feasible, but likely not in polynomial time (\cite{gap}).

When there are more than two graphs, the problem has been less extensively studied. Several studies have focused on developing algorithms (\cite{multi3}, \cite{multi2}, \cite{multi1}); a few studies (\cite{exacter}, \cite{newah1}) have also looked at the information-theoretical side of this problem. Here, we pursue this information-theoretical line of inquiry; \cite{newah2} is the most closely related to our present work, and we will contrast its contributions with ours right after the statement of Theorem \ref{mainthm} below. 
\subsection{Contributions}

Let us state the main results of this paper. {If $\bi, \bi' \in \{ \Id \} \times \mc{S}_n^{p-1}$, we will define 

\begin{equation} \label{overlap}
    ov(\bi, \bi') = \frac{1}{n} \sum_{i=1}^n \1_{\pi_2(i) = \pi'_2(i), \ldots, \pi_p(i) = \pi'_p(i)}.
\end{equation}
}
We begin by examining whether alignment is possible for the Gaussian model.
\begin{theorem}\label{mainthm}
    Let $c_p = \dfrac{8}{p}$; there is a cutoff for graph alignment at $\rho_0 = \sqrt{\dfrac{c_p \log n}{n}}$.

Specifically, for any $\eps > 0$:
\begin{itemize}
    \item if $\rho \geq (1 + \eps) \rho_0$, there exists an estimator $\hat{\bi}$ which is equal to the underlying permutation $\bi^*$ with high probability;
    \item if $\rho \leq (1 - \eps) \rho_0$, partial alignment is intractable: for any estimator $\hat{\bi}$ of $\bi^*$, $ov(\hat{\bi}, \bi^*) = o_{\Proba}(1)$.
\end{itemize}
\end{theorem}
As in the $2$-graph case (see \cite{settling}), we observe an "all-or-nothing" phenomenon: either exact alignment is possible, or partial alignment is intractable. This is a broad phenomenon, which has been observed for many classes of dense models; we thus see that it still holds for $p \geq 3$ graphs.

 Following this document's prepublication, the authors were made aware of the concurrent study \cite{newah2} studying the closely related {detection problem} for the Gaussian case (i.e. determining whether the graphs are correlated); it shows that detection is possible above the threshold $\rho_0$, and is impossible underneath the threshold $\frac{\rho_0}{2}$. The second part of our theorem suggests that the true threshold for this problem may also be $\rho_0$, but this is not immediately clear.

We will also formulate a similar statement for the Erdős–Rényi model.
\begin{theorem} \label{erthm}
    Assume that $\lambda s(1 - (1 - s)^{p-1}) < 1$. Then, partial alignment is intractable.
\end{theorem}

Just like in the two-graph case, this hypothetical cutoff represents a topological phenomenon: $\lambda s(1 - (1 - s)^{p-1})$ is the parameter of the Erdős–Rényi graph $H^{(1)} = G^{(1)} \cap (G^{(2)} \cup \ldots \cup G^{(p)})$, and so our theorem-conjecture pair states that the feasibility of partial alignment should depend only on whether the graph $H^{(1)}$ has a giant component. This idea is supported by the results from the recent preprint \cite{exacter}, where it is proven that the cutoff for exact alignment is the same as the cutoff for $H^{(1)}$ to be connected.

The proofs of these theorems are based upon a common framework, which we develop in Section \ref{infsection}: we establish some necessary and sufficient conditions for the feasibility of a certain class of Bayesian estimation problems (including partial graph alignment). This framework is applicable to a large number of high-dimensional statistical problems, beyond simply graph alignment problems.

{ How about the feasibility threshold for the Erdős–Rényi model? Well, we believe that Theorem \ref{erthm} is, in fact, sharp:

\begin{conjecture} \label{erconj}
     Assume that $\lambda s(1 - (1 - s)^{p-1}) > 1$. Then, partial alignment is feasible: there exists an estimator $\hat{\bi}$ and some $\eps > 0$ such that, with high probability, $ov(\hat{\bi}, \bi^*) > \eps$.
\end{conjecture}

This seems harder to prove than our intractability statement even if it is the same threshold. For our theorem, we simply need to exhibit an obstruction to alignment: finding any symmetries of the problem which make alignment intractable is sufficient, and this is exactly what we do in Proposition \ref{hamiltonianinvariant}. However, if we wished to tackle the conjecture, this would amount to showing that no such symmetries exist. The method used to show Conjecture \ref{erconj} in the case $p=2$, in the paper \cite{infotheory}, uses cycle decompositions to obtain uniform controls over $\sigma \in \mc{S}_n$ and prove that a certain estimator succeeds; in our case, due to the overlapping structures of multiple permutations, cycle decompositions are not possible and the method fails.}

Finally, we will note that during the proof of Theorem \ref{mainthm}, we also prove a lower bound on the weight of the maximal spanning tree of certain weighted graphs, which appears to be novel (to our knowledge) and may be of independent interest. The specific result is detailed in Theorem \ref{spanthm} and Remark \ref{spanrk}.

\subsection{Structure of the paper and notation}

In section \ref{infsection}, we lay out the aforementioned framework for analysing a certain class of Bayesian estimation problems. In Section \ref{gaussproof}, we prove Theorem \ref{mainthm} regarding the Gaussian model; in Section \ref{erproof}, we prove Theorem \ref{erthm} regarding the Erdős–Rényi model.

Throughout this paper, $n$ represents the size of the considered graphs; all limits are taken as $n \to +\infty$. For random variables $X_n, Y_n$, we write $X = O_{\Proba}(Y)$ (resp. $X = o_{\Proba}(Y)$) to mean that $\frac{X_n}{Y_n}$ is tight (resp. $\frac{X_n}{Y_n} \toPr 0$). We write $x_n \ll y_n$ to mean $x_n = o(y_n)$, and we define $\ll_{\Proba}$ similarly. We say that a sequence $(A_n)$ of events occurs with high probability if $\Proba(A_n) \to 1$.
$E$ is the set of edges (i.e. unordered pairs of integers) in $\llb 1, n  \rrb$. 

An alignment $\bi$ of $p$ graphs $(G^{(1)}, \ldots, G^{(p)})$ is a $p$-tuple $(\pi_1 = \Id, \ldots, \pi_p)$, where $\pi_i$ is a bijection between the vertices of $G^{(i)}$ and those of  $G^{(1)}$. Since the vertex set of $G^{(i)}$ is assumed to be $\llb 1, n \rrb$, the set of alignments can be identified with $\{ \Id \} \times \mathcal{S}_n^{p-1} \simeq \mathcal{S}_n^{p-1}$, which we will do for the rest of this paper. 

For any matrix $M$, $\|M\|_F$ is its Frobenius norm and $\|M\|_{op}$ is the associated operator norm. $I_k$ is the identity matrix of size $k \times k$; $J_k$ is the all-ones matrix of size $k \times k$.  If $A$ is a set, we denote by $\mathcal{P}(A)$ the set of subsets of $A$.

Proofs of the various statements made below can be found in the appendix if they are not in the main body of text.

\section{Bayesian estimation over metric spaces} \label{infsection}
\subsection{The graph alignment problem as a Bayesian inference problem}
The graph alignment problem is a Bayesian inference problem. As defined previously, it involves sampling $p$ random graphs $(G^{(1)}, G^{(2)}, \ldots, G^{(p)})$ from some fixed joint distribution; an alignment $\bi^* \in \mathcal{S}_n^{p-1}$ is then picked uniformly at random, and we try to recover $\bi^*$ from the observation of the graphs $(\bi^*)^{-1}(G) = ((\bi^*_1)^{-1}(G^{(1)}), (\bi^*_2)^{-1}(G^{(2)}), \ldots, (\bi^*_p)^{-1}(G^{(p)}))$. Another way to formulate this is that we have observed a $p$-tuple of random graphs $\tilde{G} =(\tilde{G}^{(1)}, \ldots, \tilde{G}^{(p)})$ (where $\tilde{G} = (\bi^*)^{-1}(G)$), sampled from a distribution $\Proba_{\bi^*}$ for some unknown $\bi^* \in \mathcal{S}_n^{p-1}$ whose prior distribution is assumed to be the uniform distribution $\mathcal{U}(\mathcal{S}_n^{p-1})$; again, we wish to estimate $\bi^*$. This reformulation of the problem places it squarely within the zoo of Bayesian inference problems. Specifically, we are trying to construct an estimator $\hat{\bi}$, which is a  $\tilde{G}$-measurable random variable, such that the expected loss
\begin{equation}
    L = \E_{\bi^* \sim \mathcal{U}(\mathcal{S}_n^{p-1})}[\E_{(\tilde{G}^{(1)}, \ldots, \tilde{G}^{(p)}) \sim \Proba_{\bi^*}}[l(\bi^*, \hat{\bi})]]
\end{equation}
is as small as possible, where $l: (\mathcal{S}_n^{p-1})^2 \to \R_+$ is a predetermined loss function.

What would be an adequate choice for $l$? A natural default choice could be the all-or-nothing loss function $l(\bi, \bi') = \1_{\bi \neq \bi'}$; with this loss function, this inference problem is known as \textbf{exact alignment}. However, we have a notion of closeness for permutations, and we would like to acknowledge estimators which are "close enough" to the true alignment. 

In the case where $p=2$, we may define
\begin{equation}\label{eq:overlap}
    ov(\pi, \pi') = \frac{1}{n} \sum_{i=1}^n \1_{\pi(i) = \pi'(i)} \in [0, 1]
\end{equation}
and $d = (1 - ov)$ is a distance over $\mathcal{S}_n$. This means that we now have other natural loss functions: for instance, the functions $l_r(\pi, \pi') = \1_{d(\pi, \pi') \geq r}$, for $0 \leq r < 1$. This allows us to ask slightly different questions: namely, the so-called \textbf{almost exact} and \textbf{partial} alignment problems, which we shall describe in the next section. We may similarly extend the notion of overlap to $\mathcal{S}_n^{p-1}$ when $p \geq 3$; there are multiple ways to do this, including the definition (\ref{overlap}) which we adopted to state our main results. We discuss this further in Section \ref{multiov}.

Before proceeding, let us describe what the metric space $\mathcal{S}_n$ actually looks like. One way to conceptualise it is as a hyperbolic space: for any $\pi \in \mathcal{S}_n$, the size of the ball $B(\pi, r)$ grows roughly like $(n^n)^r$, for $0 \leq r \leq 1$.

Another, more explicit, description of this space is the following. Within the space $\mathcal{M}_n(\R)$ of $n \times n$ matrices, the set $B_n$ of doubly stochastic matrices is a polytope (known as the Birkhoff polytope), whose vertices are the permutation matrices (which we may identify with $\mathcal{S}_n$). The canonical Euclidean structure on this space gives us the metric 

    \begin{equation}
        \langle \pi, \pi' \rangle = ov(\pi, \pi') \text{ i.e. } d_B(\pi, \pi') = \sqrt{2(1 - ov(\pi, \pi'))}
    \end{equation}
    which largely serves the same purpose as our previous distance $d$ in this context. This description also shows quite nicely why we often relax the alignment problem to solve a more general quadratic assignment problem over $B_n$, as in \cite{fan2023}, \cite{varma2025}.

\subsection{Metric estimation problems} \label{estsection}

What follows is fairly general; as motivation for our definitions, it may be helpful to think of the specific case of the graph alignment problem, as defined above.

Fix a sequence $(A_n, d = d_n, \mu = \mu_n)$, which are metric spaces equipped with an a priori probability distribution. Fix a sequence $(B_n)$ of (measurable) observation spaces, equipped with probability distributions $(\nu_x)_{x \in A_n}$ such that $x \mapsto \nu_x$ is measurable. To fix ideas, we will assume that $\diam A_n = 1$, so that $d(x, y) \simeq 1$ means that $x, y$ are far apart.

We set $X_n \in A_n$ a random variable with law $\mu$, and $Y_n \in B_n$ a random variable with law $\nu_{X_n}$; from the observation of $Y_n$, we are trying to estimate the value of $X_n$. We define $\Proba = \Proba_n$ to be the distribution of $(X_n, Y_n)$, and, for any measurable $S_n \subseteq A_n$,
\begin{equation}
    \Proba_{post}(S_n) = \Proba(S_n | Y_n)
\end{equation}
the a posteriori probability of $S_n$.
\begin{definition} \label{posstract}
    A \textbf{partial estimator} $\hat{X}_n$ of $X_n$ is a $Y_n$-measurable random variable such that, for some $0 \leq r < 1$,
        \begin{equation}
            \E[l_r(\hat{X}_n, X_n)] \underset{n \to +\infty}{\longrightarrow} 0 \hspace{10pt} \text{i.e.} \hspace{10pt}  \Proba(d(\hat{X}_n,  X_n) > r) \underset{n \to +\infty}{\longrightarrow} 0
        \end{equation}
        We will say that partial estimation is \textbf{feasible} if a partial estimator exists.
        
        On the other hand, we will say that partial estimation of $X_n$ is \textbf{intractable} if, for any $0 \leq r < 1$, and for any sequence $(\hat{X}_n)$ of estimators,
        \begin{equation}
            \E[l_r(\hat{X}_n, X_n)] \underset{n \to +\infty}{\longrightarrow} 1
        \end{equation}
         An \textbf{almost exact estimator} $\hat{X}_n$ of $X_n$ is a $Y_n$-measurable random variable such that, for any $\eps > 0$, 
        \begin{equation}
            \E[l_\eps(\hat{X}_n, X_n)] \underset{n \to +\infty}{\longrightarrow} 0  \hspace{10pt} \text{i.e.} \hspace{10pt} \Proba(d(\hat{X}_n,  X_n) > \eps) \underset{n \to +\infty}{\longrightarrow} 0
        \end{equation}
        We will say that almost exact estimation is \textbf{feasible} if an almost exact estimator exists.
        
        On the other hand, we will say that almost exact estimation of $X_n$ is \textbf{intractable} if, for some $\eps > 0$, and for any sequence $(\hat{X}_n)$ of estimators,
        \begin{equation}
            \E[l_\eps(\hat{X}_n, X_n)] \underset{n \to +\infty}{\longrightarrow} 1
        \end{equation} 
\end{definition}
{
\begin{remark}
    Notice that partial (resp. almost exact) estimation being intractable is a stronger statement than the nonexistence of a partial (resp. almost exact) estimator. Indeed, there exists an intermediate regime where one could build estimators which - say - succeed half the time and fail half the time. 
\end{remark}
}
A classical result from Bayesian decision theory (see \citep[Section~4]{berger} for instance) is the following.
\begin{proposition} \label{classic}
    Fix $r \geq 0$: we define
    \begin{equation}
        \hat{X}_n^{opt} = \underset{x \in A_n}{\arg \min}\E[l_r(x,X_n) | Y_n] = \underset{x \in A_n}{\arg \max} \Proba_{post}(B(x, r))
    \end{equation}
    where $B(x, r)$ is the (closed) ball of radius $r$ centred at $x$, for the distance $d$. Then, $\hat{X}_n^{opt}$  minimises the expected loss $\E[l_r]$. 
\end{proposition}

To use this, we will define a random variable

\begin{equation}
	C_n(r) \defeq \max_{x \in A_n} \Proba_{post}(B(x, r)) = 1 - \E[l_r(\hat{X}_n^{opt}, X_n) | Y_n] 
\end{equation}

which measures to what extent the random measure $\Proba_{post}$ concentrates around a single point. Then, owing to the choices made in defining feasibility and intractability, the following result is an immediate consequence of the proposition.

\begin{corollary}
     \label{estimation}
        With $C_n$ defined as above,
    \begin{enumerate}
        \item partial estimation is feasible iff, for some $0 \leq r < 1$, $C_n(r) \underset{\Proba}{\longrightarrow} 1$.
        \item almost exact estimation is feasible iff, for any $\eps > 0$, $C_n(\eps) \underset{\Proba}{\longrightarrow} 1$. \\ 
    \end{enumerate}
    On the other hand,
    \begin{enumerate}
        \item partial estimation is intractable iff, for any $0 \leq r < 1$, $C_n(r) \underset{\Proba}{\longrightarrow} 0$.
        \item almost exact estimation is intractable iff, for some $\eps > 0$, $C_n(\eps) \underset{\Proba}{\longrightarrow} 0$.
    \end{enumerate} 
\end{corollary}

\begin{remark}
Assume that almost exact estimation is feasible, and let $\eps > 0$. Then,

\begin{equation}
    \Proba_{post}(B(X_n, 2 \eps)) \geq \Proba_{post}(B(\hat{X}_n^{opt}, \eps) )\1_{d(X_n, \hat{X}_n^{opt}) < \eps} = 1 - o_{\Proba}(1).
\end{equation}

Otherwise put, with high probability, the mass of $\Proba_{post}$ will tend to concentrate around $X_n$.
\end{remark}

\subsection{A brief aside on multi-graph alignment} \label{multiov}
As we previously suggested, there are a few natural extensions of $ov$ to $\mathcal{S}_n^{p-1}$. Two particularly relevant ones here are: 
\begin{equation} \label{strongov}
   ov(\bi, \bi') = \frac{1}{n} \sum_{i=1}^n \1_{\pi_2(i) = \pi'_2(i), \ldots, \pi_p(i) = \pi'_p(i)} \text{ (as defined earlier)}
\end{equation}
\begin{equation}
   ov_w(\bi, \bi') = \frac{1}{p-1} \sum_{i=2}^p ov(\pi_i, \pi'_i)
\end{equation}
Essentially, $ov$ measures how often the $(\pi_i, \pi_i')$ all agree at the same time; while $ov_w$ measures how often they will agree pairwise. Thus, informally, graph alignment for $d = 1 - ov$ involves determining $(\pi_2^*, \ldots, \pi_p^*)$ on some given set $X$, while graph alignment for $d_w = 1 - ov_w$ involves determining $\pi_2^*$ on some (large) set $X_2$, \ldots, and $\pi_p^*$ on some (large) set $X_p$.  With this in mind, it seems that the distance $d$ should be more natural as a choice, motivating the following definition.

\begin{definition}
    We will say that partial graph alignment is possible (resp. intractable) if partial estimation of the true alignment $\bi^*$, for the distance function $d$, is possible (resp. intractable), as defined in Definition \ref{posstract}.
    For the distance $d_w$, we will instead speak of \textbf{weak} partial alignment.
\end{definition}

The same definition applies to almost exact alignment. However, in this case, we do not need to make the distinction between normal and weak alignment:

\begin{proposition} If partial alignment is possible (resp. not intractable), then weak partial alignment is possible (resp. not intractable). 

Almost exact alignment is possible (resp. intractable) if and only if  weak almost exact alignment is possible (resp. intractable).
\end{proposition}

\begin{proof}
    This follows from the fact that $d_w \leq d \leq (p - 1) d_w$. \qed
\end{proof}

\subsection{Applying metric estimation} \label{outline}

If we want to understand partial, or almost exact, alignment, our job is thus to control the random variable $C_n(r)$ defined earlier. In the Gaussian case, in order to show intractability of partial alignment, we can directly perform calculations to estimate $C_n(r)$.  

In the Erdős–Rényi case, however, our arguments are much more combinatorial in nature. The general idea is quite simple: if $B = B(\bi^*, r)$ is a ball, we would like to take the alignments within $B$ (which are close to the true alignment $\bi^*$) and modify them to obtain a large number of other alignments, which are nowhere near $\bi^*$ but which seem just as plausible to an outside observer. In the case where $p=2$, this is done in \cite{partialimposs}: they construct a large number of $(\sigma_i) \in S_n$ which are far apart and which \textbf{exactly} preserve $\Proba_{post}$, so that $\Proba_{post}(\sigma_i(B)) = \Proba_{post}(B)$ and the $\sigma_i(B)$ are all disjoint. This is fairly technical, and not quite necessary: at the cost of some abstraction, we can provide a significantly more robust argument, which holds for any metric estimation problem, and does not require that we preserve $\Proba_{post}$ exactly.
\begin{lemma} \label{invariance}
    Assume that the state space $A_n$ is finite. Let $\Neg_n \subseteq A_n$, such that $\Proba_{post}(\Neg_n) = o_\Proba(1)$. (This contains states which are implausibly ill-behaved and which we may ignore.)
    
    For any ball $B = B(x_0, r)$ with radius $r < 1$, assume that there exists a random function $F: B \setminus \Neg_n \to \mathcal{P}(A_n)$ (which depends upon $n, B$), satisfying certain conditions:
    \begin{enumerate}[label=(\roman*)]
        \item $F$ outputs an overwhelmingly plausible set of states, in the following sense: there exists $\eps_n(r) > 0$ such that, for $x \in B \setminus \Neg_n$,
        \begin{equation}
            \Proba_{post}(x) \leq \eps_n(r) \Proba_{post}(F(x)).
        \end{equation}    
        \item $F$ is quasi-injective, in the following sense. For any $r < 1$, there exists  $K_n(r) > 0$ such that $K_n(r) \eps_n(r) \overset{\Proba}{\underset{n \to +\infty}{\longrightarrow}} 0$ and, if $y \in A_n$, there are at most $K_n(r)$ values of $x \in B$ such that $y \in F(x)$.
    \end{enumerate}
    Then, partial estimation is intractable.
\end{lemma}

\begin{remark}
    Notably, we do not require that the quantities $\eps_n, K_n$ be deterministic: this would be somewhat restrictive given that the measure $\Proba_{post}$ is itself random.
\end{remark}

Here, $F$ takes on the same role as the map $\pi \mapsto (\sigma_i \circ \pi)_{i}$, in the $2$-graph argument from \cite{partialimposs}.

\begin{proof}
Let $B = B(x, r)$ be a ball with radius $r < 1$. Then:
    \begin{equation}
    \begin{split}
        \Proba_{post}(B) & \leq \Proba_{post}(\Neg_n) + \sum_{y \in B \setminus \Neg_n} \Proba_{post}(y) \\
        & \leq o_{\Proba}(1) + \eps_n(r)\sum_{y \in B \setminus \Neg_n}  \Proba_{post}(F(y)) \\
        & \leq o_{\Proba}(1) + \eps_n(r) \sum_{y' \in A_n} \sum_{y \text{ s.t. }y' \in F(y)} \Proba_{post}(y') \\
        & \leq o_{\Proba}(1) + \eps_n(r) \sum_{y' \in A_n} K_n(r) \Proba_{post}(y') \\
        & = o_{\Proba}(1) + \eps_n(r) K_n(r) \Proba_{post}(A_n) = o_{\Proba}(1)
    \end{split}
\end{equation}

Applying Corollary \ref{estimation}, the lemma is shown.\qed
\end{proof}
\section{Solving the Gaussian model} \label{gaussproof}
In this section, we prove Theorem \ref{mainthm}. As discussed in Section \ref{multiov}, it will be sufficient for us to prove it for the distance function $d$; from now on, all balls $B(\bi, r)$ are defined with regards to this metric.

In all of the following theorems, we assume that 
\begin{equation}
    \rho = \ds\sqrt{\frac{8(1 + \eta) \log n}{p n}}
\end{equation}
for some $\eta = O(1)$. We will prove the feasibility of exact alignment for any fixed $\eta > 0$, and show that partial alignment is intractable when $\eta = -(\log n)^{- \frac{1}{5}}$. 

Note that this is, in fact, sufficient to prove the theorem for general $\rho$. Indeed, if we observe $(\bi^*)^{-1}(G)$ for some $G \sim \mathcal{N}(n, \rho)$, we may form a set of graphs $(\bi^*)^{-1}(G')$ where $G' \sim \mathcal{N}(n, \rho')$ for any $\rho' \leq \rho$, simply by adding i.i.d. Gaussian weights to the edges of $(\bi^*)^{-1}(G)$ and renormalising; then, solving the alignment problem for $G'$ allows us to solve it for $G$.

In what follows, for $\bigma \in \mathcal{S}_n^{p-1}$ and $1 \leq i \neq j \leq p$, we will set $\sigma_{ij} = \sigma_j (\pi_j^*)^{-1} \pi_i^* \sigma_i^{-1} $. It may be helpful to assume that $\boldsymbol{\pi}^* = \Id$, since we can (up to some notational bookkeeping) reduce the problem down to this case; then, $\sigma_{ij}$ measures how different $\sigma_i$ and $\sigma_j$ are.

\subsection{The Gibbs measure}
The main appeal of the Gaussian model is that the a posteriori distribution has a nice expression. Specifically:

\begin{proposition} \label{gibbs}
    For any $\bigma \in \mathcal{S}_n^{p-1}$, we may write

    \begin{equation}
        \Proba_{post}(\bigma) = \frac{1}{\prt} e^{- \beta \mathcal{H}(\sigma)}
    \end{equation}
    where the partition function $\prt$ is independent from $\bigma$. Here,

    \begin{equation}
        \mathcal{H}(\bigma) = - \sum_{e \in E} \sum_{1 \leq i \neq j \leq p} G_e^{(i)} G_{\sigma_{ij}(e)}^{(j)}; \hspace{5pt} \beta = \dfrac{\rho}{2(1 - \rho)(1 + (p-1)\rho)} = \dfrac{\rho}{2}(1 + o(1))
    \end{equation}

\end{proposition}

\begin{proof}
    By definition, the probability density function for $(\bigma, G^{(1)}, \ldots, G^{(p)})$ is proportional to

    \begin{equation}
                p(\bigma, G^{(1)}, \ldots, G^{(p)})  = \frac{1}{(n!)^{p-1}} \prod_{e \in E} \exp\left( -\frac{1}{2}\sum_{1 \leq i,j \leq p} (M^{-1})_{ij} G^{(i)}_{e} G^{(j)}_{e}\right)
    \end{equation}
    where $M = (1 - \rho)I_p + \rho J_p$ is our covariance matrix. Consequently, if we observe the graphs $H^{(i)} = (\pi_i^*)^{-1}(G^{(i)})$, the distribution $\Proba_{post}$ is given by

    \begin{equation}
        \begin{split}
            \Proba_{post}(\bigma) & = \Proba(\boldsymbol{\pi}^* = \bigma | H^{(1)}, \ldots, H^{(p)}) \\
            & = \frac{1}{\prt'} p(\bigma, \sigma_1(H^{(1)}), \ldots, \sigma_p(H^{(p)})) \\
            & = \frac{1}{\prt'} \prod_{e \in E} \exp\left( -\frac{1}{2}\sum_{1 \leq i,j \leq p} (M^{-1})_{ij} G^{(i)}_{(\pi_i^*)^{-1} \sigma_i(e)} G^{(j)}_{(\pi_j^*)^{-1} \sigma_j(e)}\right) \\
            & = \frac{1}{\prt} \exp \left( - \frac{1}{2}\sum_{e \in E} \sum_{1 \leq i \neq j \leq p} (M^{-1})_{ij} G^{(i)}_{(\pi_i^*)^{-1} \sigma_i(e)} G^{(j)}_{(\pi_j^*)^{-1} \sigma_j(e)}\right)
        \end{split}
    \end{equation}

    since the product over the diagonal terms ($i=j$) is independent from $\bigma$. (Here, $\mc{Z}, \mc{Z'}$ are quantities which may depend upon the $(H^{(i)})$ but not upon $\bigma$.) However, for $i \neq j$, $(M^{-1})_{ij} = - 2 \beta$; we thus obtain our desired expression, with

    \begin{equation}
        \begin{split}
            \mathcal{H}(\sigma) & = -\sum_{e \in E} \sum_{1 \leq i \neq j \leq p}  G^{(i)}_{(\pi_i^*)^{-1} \sigma_i(e)} G^{(j)}_{(\pi_j^*)^{-1} \sigma_j(e)} \\
            & = -\sum_{e \in E} \sum_{1 \leq i \neq j \leq p} G^{(i)}_{e} G^{(j)}_{\sigma_{ij}(e)}
        \end{split}
    \end{equation} \qed
\end{proof}

This Hamiltonian is quite descriptive, and invites a natural visual interpretation of the graph alignment problem: we have $p$ weighted graphs which we "stick" together with our alignment $\bigma$, and each edge has a fixed "spin" such that edges with similar spins try to stick to each other. 

From a technical perspective however, this Hamiltonian is not quite what we want. Recall that our ultimate goal is to understand to what extent the measure $\Proba_{post}$ concentrates around any given point; and, if it is to be concentrated around a point, we should expect that point to be the ground truth $\pi^*$. As a result, it makes sense to renormalise our system, writing

\begin{equation}
    \Proba_{post}(\bigma) = \frac{1}{Z} e^{- \beta V(\bigma)}
\end{equation}

where the potential $V(\bigma)$ is equal to
\begin{equation}
\begin{split}
    V(\bigma)  & = \mathcal{H}(\bigma) - \mathcal{H}(\pi^*) \\
    & = \sum_{e \in E, 1 \leq i \neq j \leq p} G_e^{(i)} (G_e^{(j)} - G_{\sigma_{ij}(e)}^{(j)}) \\
    & = \sum_{\underset{\sigma_{ij}(e) \neq e}{e \in E, 1 \leq i \neq j \leq p}}G_e^{(i)} G_e^{(j)} - \sum_{\underset{\sigma_{ij}(e) \neq e}{e \in E, 1 \leq i \neq j \leq p}}G_e^{(i)} G_{\sigma_{ij}(e)}^{(j)} \\
    & \defeq - V_{\mathrm{diag}}(\bigma) + V_{\mathrm{off}}(\bigma)
\end{split}
\end{equation}
We have split up the potential $V$ into two terms for a couple of reasons. The first one is that, viewing $V$ as a quadratic form of $G \in \R^{E} \otimes \R^p$, $V_{\mathrm{diag}}$ contains the block-diagonal terms while $V_{\mathrm{off}}$ contains the off-diagonal blocks. (The expression of $V$ as a quadratic form is made explicit in Appendices \ref{quad} and \ref{quad2}.) It is typical, when studying quadratic forms of Gaussian vectors, to separate the two; they tend to require different tools. The second reason is that $V_{\mathrm{diag}}(\bigma)$ concentrates around its mean quite uniformly in $\bigma$, in a sense which we will make explicit in a moment.

Before moving on, let us briefly explain the structure of the rest of this section. Our ultimate goal is to prove that $\Proba_{post}$ concentrates around the true alignment $\bi^*$ (resp. is spread out over the entire state space) if $\eta > 0$ (resp. $\eta < 0$). Doing this will require a few steps:

\begin{itemize}
    \item first, we will do a bit of combinatorics in order to understand the structure of $(\mathcal{S}_n^{p-1}, d)$ a bit better;
    \item next, we shall compute the "annealed free energy" $\log \E[Z]$; 
    \item finally, we will use this to compute the quenched free energy $\E[\log Z]$.
\end{itemize}

This last step is the harder one; as is fairly standard in systems derived from Gibbs measures, once we have $\E[\log Z]$ it is not too hard to obtain most relevant properties of the distribution $\Proba_{post}$.

\subsection{Some combinatorics}
For the rest of this section, we will define $d_{ij}(\bigma) =\# \{ u \in \llb 1, n \rrb,  \sigma_{ij}(u) = u \}$: we will see that this combinatorial quantity appears a fair amount in our analysis, as it is closely linked to the number of terms in the sums defining $V_{\mathrm{diag}}$ and $V_{\mathrm{off}}$. In this section, we control the behaviour of $(d_{ij}(\bigma))$ as a function of $\bigma \in \mathcal{S}_n^{p-1}$.

\begin{theorem} \label{permcount}
    Let $(d_{ij})_{1 \leq i \neq j \leq p}$ be positive integers, with $d_{ij} = d_{ji}$ for any $i,j$.  We set

    \begin{equation}
        F((d_{ij})) = \# \{\bigma \in \mathcal{S}_n^{p-1}, \forall 1 \leq i \neq j \leq p, d_{ij}(\bigma) \geq d_{ij} \}
    \end{equation}
    
    Then, for any permutation $\tau \in \mathcal{S}_p$,

    \begin{equation}
                F((d_{ij})) \leq \frac{(n!)^{p-1}}{(\ds\max_{j < 2} d_{\tau(j)\tau(2)})!(\ds\max_{j < 3} d_{\tau(j) \tau(3)})! \ldots (\ds\max_{j < p} d_{\tau(j) \tau(p)})!}
    \end{equation}
\end{theorem}

\begin{proof}
It is sufficient to handle the case where $\boldsymbol{\pi}^* = \Id$; the general case follows by replacing $\bigma$ by $\boldsymbol{\pi}^* \circ \bigma$ in the following proof.

We will begin by noting that $F$ is a symmetric function of $(d_{ij})$, in the sense that if $\tau \in \mathcal{S}_p$, $F((d_{\tau(i)\tau(j)})) = F((d_{ij}))$. As a result, we just need to show the theorem when $\tau = \Id$; the general statement then follows.
    
    We will now note that, for any $(d_{ij})$, setting $Y_{ij}(\bigma) \defeq \{ u \in \llb 1, n \rrb, \sigma_{ij}(u) = u \}$,

    \begin{equation} \label{puples}
        F((d_{ij})) = \frac{1}{n!}\# \{\bigma \in \mathcal{S}_n^{p}, \forall 1 \leq i \neq j \leq p, |Y_{ij}(\bigma)| \geq d_{ij} \}
    \end{equation}

    (i.e. counting the $p$-tuples of permutations where we do \textbf{not} enforce that $\sigma_1 = \Id$). To see this, note that the map 

    \begin{align*}
        \mathcal{S}_n^{p} & \longrightarrow  \mathcal{S}_n^{p-1} \simeq \{ \Id \} \times \mathcal{S}_n^{p-1} \\
    (\sigma_1, \sigma_2, \ldots, \sigma_p) & \longmapsto (\Id, \sigma_1^{-1} \sigma_2, \ldots, \sigma_1^{-1} \sigma_p)
    \end{align*}

    preserves the values of $(d_{ij})_{1 \leq i \neq j \leq p}$, and any given element of $\mathcal{S}_n^{p-1}$ has exactly $n!$ preimages. We will therefore count the number of elements of the subset defined in (\ref{puples}).

    Fix $(d_{ij})_{i\neq j}$. We are going to bound the number of possible choices for $\sigma_1$, then bound the number of possible choices for $\sigma_{2}$ given a choice of $\sigma_{1}$ and so on; finally bounding the number of choices for $\sigma_{p}$ given a choice of $\sigma_{1}, \ldots, \sigma_{p-1}$.

    There are $n!$ choices for $\sigma_{1}$. In order to fix $\sigma_{2}$ given $\sigma_{1}$, it is sufficient to determine the set $X_2 = \llbracket 1, n \rrbracket \setminus Y_{1 2}(\sigma)$ of points where $\sigma_{2} \neq \sigma_1$, as well as $\sigma_{2}|_{X_2}^{\sigma_1(X_2)}$, giving at most 

    \begin{equation}
        \binom{n}{n - d_{12}}(n - d_{12})! = \frac{n!}{d_{12}!}
    \end{equation}
    possible choices.

    Recursively, given $\sigma_{1}, \sigma_{2}, \ldots, \sigma_{l-1}$, in order to fix $\sigma_{l}$, it is sufficient to choose $j < l$, and determine the set $X_{l} = \llbracket 1, n \rrbracket \setminus Y_{jl}(\sigma)$ of points where $\sigma_{l} \neq \sigma_{j}$ as well as the restriction $\sigma_{l}|_{X_{l}}^{\sigma_{j}(X_{l})}$, giving at most $\frac{n!}{d_{jl}!}$ possibilities. Since $j$ can be chosen arbitrarily, we can pick $j$ such that $d_{jl}$ is maximal. The total number of possible choices is thus at most

    \begin{equation}
        \frac{1}{n!} \cdot n! \cdot \frac{n!}{(\ds\max_{j < 2} d_{j2})!} \cdot \ldots \cdot \frac{n!}{(\ds\max_{j < p} d_{jp})!}
    \end{equation}
    giving us the desired inequality and concluding the proof of Theorem \ref{permcount}.
\qed
\end{proof}

Theorem \ref{permcount} can be used by itself, but the introduction of the quantities $\max_{j<k} d_{\tau(j)\tau(k)}$ for $1 < k \leq p$ makes it somewhat unwieldy to use. Instead, we will mostly use the following corollary.

\begin{corollary}
 \label{usable}
     Let $(d_{ij})_{1 \leq i \neq j \leq p}$ be positive integers, with $d_{ij} = d_{ji}$ for any $i,j$. Then, 
    
    \begin{equation} \label{usablef}
        F((d_{ij})) \leq \frac{(n!)^{p-1}}{\left( \ds\prod_{1 \leq i \neq j \leq p} (d_{ij}!) \right)^{\frac{1}{p}}}
    \end{equation}
\end{corollary}
\begin{proof}
For $\tau \in \mathcal{S}_p$, set $D_{i}(\tau) = \max_{j < i} d_{\tau(j)\tau(i)}$. In order to deduce Corollary \ref{usable} from Theorem \ref{permcount}, we need to show that, for some choice of $\tau$,

\begin{equation}
    \prod_{i=2}^p D_i(\tau)! \geq \left( \prod_{1 \leq i \neq j \leq p} (d_{ij}!) \right)^{\frac{1}{p}}
\end{equation}

In order to show this, we will use the probabilistic method, showing that if $\tau \in \mathcal{S}_p$ is chosen uniformly at random, then

\begin{equation} \label{?}
     \frac{1}{p} \sum_{1 \leq i \neq j \leq p} \log (d_{ij}!) \leq \E_\tau\left[ \sum_{i=2}^p \log (D_i(\tau)!) \right]
\end{equation}

Indeed, more generally, if $\phe$ is any increasing function,

\begin{equation} \label{yeah}
    \begin{split}
        \E_\tau\left[ \sum_{i=2}^p \phe(D_i(\tau)) \right] & = \E_\tau\left[ \sum_{i=2}^p \max_{1 \leq j \leq i-1} \phe(d_{\tau(i) \tau(j)}) \right]\\
        & \geq \E_\tau\left[ \sum_{i=2}^p \frac{1}{i-1} \sum_{j=1}^{i-1} \phe(d_{\tau(i) \tau(j)}) \right] \\
        & = \sum_{i=2}^p \frac{1}{i-1} \sum_{j=1}^{i-1} \E_\tau[ \phe(d_{\tau(i) \tau(j)})]
    \end{split}
\end{equation}

However, for any $i,j \in \llb 1, p \rrb$, $\E_\tau[\phe(d_{\tau(i) \tau(j)})] = \dfrac{1}{p(p-1)}\ds \sum_{1 \leq i \neq j \leq p} \phe(d_{ij})$. Plugging this into (\ref{yeah}), we obtain

\begin{equation}
     \frac{1}{p} \sum_{1 \leq i \neq j \leq p} \phe(d_{ij}) \leq \E_\tau\left[ \sum_{i=2}^p \phe(D_i(\tau)) \right] 
\end{equation}

and this holds in particular for the function $\phe(d) = \log(d!)$, proving the corollary. \qed
\end{proof}

\subsection{The annealed free energy}

The propositions in this subsection are consequences of some sharp controls on the behaviour of their Gaussian components; their proofs are all deferred to the appendix.

In order to compute $\log \E[Z]$, we begin by computing $V_{\mathrm{diag}}$ and $V_{\mathrm{off}}$.
The behaviour of $V_{\mathrm{diag}}$ is the simplest to bound: it is essentially deterministic.

\begin{proposition}
\label{weaklemma}
    For some $c > 0$, the following property $(\mathcal{C})$ {holds with probability at least $1 - \exp(-c n \sqrt{\log n})$}:

    \begin{equation} \label{weakprob}
        \forall \bigma \in \mathcal{S}_n^{p-1}, \left|V_{\mathrm{diag}}(\bigma) - \E[V_{\mathrm{diag}}(\bigma)]\right| \leq  (\log n)^{\frac{1}{4}} n^{\frac{3}{2}}
    \end{equation}
    Consequently, for any $\bigma \in \mathcal{S}_n^{p-1}$,

    \begin{equation}
         V(\bigma)\1_\mathcal{C} =  \rho  \sum_{1 \leq i \neq j \leq p} \left( \binom{n}{2} - \binom{d_{ij}(\bigma)}{2} \right)\1_\mathcal{C} + V_{\mathrm{off}}(\bigma) \1_{\mathcal{C}} + O(n^{\frac{3}{2}} (\log n)^{\frac{1}{4}})
    \end{equation}
    where the error term is bounded uniformly in $\bigma$.
\end{proposition}

The fluctuations of  $V_{\mathrm{off}}$ are somewhat less nice but can still be controlled.

\begin{proposition} \label{cumulants}
    Let $\bigma \in \mathcal{S}_n^{p-1}$. Then,
    \begin{equation}
        \log \E[e^{-\beta V_{\mathrm{off}}(\bigma)}] = \beta^2\dsum_{1 \leq i \neq j \leq p} \left(\ds\binom{n}{2} - \ds \binom{d_{ij}(\bigma)}{2} \right) + O(n)
    \end{equation}
where, again, the error term is bounded uniformly in $\bigma$.
\end{proposition}

Putting these together allows us to obtain the behaviour of the annealed free energy. To this end, for any $\bi \in \mc{S}_n^{p-1}$ and $0 < r \leq 1$, we will define

    \begin{equation}\label{zr}
        Z_r(\bi) = \sum_{\bigma \in B(\bi, r)} e^{-\beta V(\bigma)}.
    \end{equation}

{
\begin{proposition} \label{energy}
    For any $0 < r \leq 1$,

    \begin{equation}
        \log \E[Z_r(\bi^*) \1_\mathcal{C}] \leq (p-1) n \log n \max_{0 \leq \rho \leq r} \phe(\rho) + O(n(\log n)^{\frac{3}{4}})
    \end{equation}
    and 

    \begin{equation}
    \begin{split}
        \log \E[(Z - Z_r(\bi^*))\1_{\mc{C}} ] & \leq (p-1) n \log n \left( \frac{p-1}{p} \max_{r' \in [0,1]} \phe(r') + \frac{1}{p} \max_{r' \in [\frac{r}{p-1}, 1] } \phe(r') \right) \\
        & + O(n (\log n)^\frac{3}{4})
        \end{split}
    \end{equation}
    where $\phe(r') = (1 + \eta)(r')^2 - (1 + 2 \eta)r'$ for $r' \in [0,1]$.

\end{proposition}

}

\subsection{The quenched free energy}

Recall that we have written the parameter $\rho$ of our system as $\rho = \sqrt{\frac{8(1 + \eta) \log n}{p n}}$; we wish to estimate the quenched free energy $\E[\log Z]$ as a function of $\eta$. This quantity captures the typical behaviour of $Z$ better than the annealed free energy; controlling it is the difficult part in bounding $\Proba_{post}$. In these systems, we cannot usually expect to have $\E[\log Z] \sim \log \E[Z]$ in general; however, when $\eta > 0$ the system fluctuates sufficiently little for this to hold.

\begin{lemma} \label{easyquenched}
    In general, we have: 
    \begin{equation}
        0 \leq \E[(\log Z)\1_{\mathcal{C}}] \leq \log \E[Z\1_{\mathcal{C}}](1 + o(1)) 
    \end{equation}
    In particular, if $\eta > 0$,

    \begin{equation}
        \E[(\log Z)\1_{\mc{C}}]= o(n \log n).
    \end{equation}
\end{lemma}

\begin{proof}

    {The first inequality simply follows from the fact that
    \begin{equation}
        Z \geq e^{-\beta V(\bi^*)} = 1
    \end{equation}

    since $V(\bi^*) = \mc{H}(\bi^*) - \mc{H}(\bi^*) = 0$.
    
     The second inequality follows from Jensen's inequality: 
     
     \begin{equation}
         \E[(\log Z)\1_\mc{C}] \leq \Proba(\mc{C}) \log \E[Z | \mc{C}] = \log \E[Z \1_{\mc{C}}](1 + o(1)).
     \end{equation}

     The final inequality follows from the previous proposition.
     \qed}
\end{proof}

The difficult part comes when $\eta < 0$: typically, we believe that $\E[\log Z]$ will be significantly smaller than $\log \E[Z]$. To get around this, we will take a small value $\eta = (\log n)^{-\frac{1}{5}}$ in the hope that the annealed free energy still provides some information in this regime. And indeed:

\begin{proposition} \label{hardquenched}
    If $\eta = -(\log n)^{-\frac{1}{5}}$, 

    \begin{equation}
        \E[\log Z] \sim \log \E[Z \1_{\mathcal{C}}] \sim (p-1)|\eta| n \log n.
    \end{equation}
\end{proposition}

\begin{proof}
    The idea is the following: we will show that
    \begin{itemize}
        \item $\log Z$ is strongly concentrated around its mean, via standard Gaussian concentration theorems;
        \item $\log Z$ has a somewhat reasonable probability of being larger than $\log \E[Z\1_{\mathcal{C}}] (1 - o(1))$
    \end{itemize}
    implying that $\E[\log Z]$ and $\log \E[Z]$ aren't too far from each other. Explicitly:

    \begin{lemma} \label{concentration}
        For some constant $c_0 > 0$, if $0 \leq K \leq n^{\frac{3}{2}}$,

        \begin{equation}
            \Proba(|\log Z - \E[\log Z]| > K) \leq 2\exp(- c_0 \frac{K^2}{n \log n}).
        \end{equation}
    \end{lemma}

    \begin{lemma} \label{technical}
        With probability at least $\exp(-O(n))$, 

        \begin{equation}
            (\log Z) \1_{\mathcal{C}} \geq (p-1)|\eta| n \log n - (p-1) n (\log n)^{\frac{3}{5}}.
        \end{equation}
    \end{lemma}

    The proof of these lemmas is deferred to the appendix. 

    With these results in hand, we may note that
\begin{equation}
\begin{split}
     \Proba((\log Z)\1_{\mathcal{C}} >  \E[\log Z] + n (\log n)^{\frac{3}{5}} ) &\leq 2\exp( -  C' n (\log n)^{\frac{1}{5}}) + \Proba(\mathcal{\overline{C}})\\
     & < \exp(- O(n)) \\
     & \leq \Proba((\log Z)\1_{\mathcal{C}} \geq (p-1)|\eta| n \log n - (p-1) n (\log n)^{\frac{3}{5}})
     \end{split}
\end{equation}

so that 

\begin{equation}
    \begin{split}
         \E[\log Z] + n (\log n)^{\frac{3}{5}}& \geq (p-1)|\eta| n \log n - (p-1) n (\log n)^{\frac{3}{5}} \\
         & \geq \log \E[Z\1_\mathcal{C}] + O(n (\log n)^{\frac{3}{4}}). \\
    \end{split}
\end{equation}
This proves the difficult part of Proposition \ref{hardquenched}. The full proposition will follow if we prove that $\E[\log Z] \sim \E[(\log Z)\1_{\mc{C}}]$ since, in this case, we may apply Lemma \ref{easyquenched} and

\begin{equation}
    \E[\log Z] \sim \E[(\log Z)\1_\mc{C}] \leq \log \E[Z \1_\mc{C}](1 + o(1)).
\end{equation}

However, if we define $t_0 > 0$ by the condition $\ds\int_{t_0}^{+\infty} \exp \left(- c_0 \frac{u^2}{n \log n} \right) du = \Proba(\overline{\mc{C}})$, then by Lemma \ref{concentration}:

\begin{equation} \label{idkgarbage}
    \begin{split}
        \E[\log Z] - \E[(\log Z)\1_\mc{C}] & = \E[(\log Z)\1_{\overline{\mc{C}}}]\\
        & \leq 2\int_{t_0}^{+\infty} (u + \E[\log Z]) \exp \left(- c_0 \frac{u^2}{n \log n} \right) du \\
        & \ll \E[\log Z].
    \end{split}
\end{equation}
This concludes the proof.
     \qed
\end{proof}

With this result, we have most of the required elements to prove Theorem \ref{mainthm}.

\begin{proof}[Proof of Theorem \ref{mainthm}]
    Assume that $\eta > 0$: we will begin by showing the existence of an almost exact estimator. In order to apply Corollary \ref{estimation}, we simply need to prove that  

    \begin{equation}
        \Proba_{post}(B(\bi^*, \eps)) \toPr 1 \text{ i.e. } Z_\eps(\bi^*) \underset{\Proba}{\sim} Z.
    \end{equation}

    This is true since, by Proposition \ref{energy},

    \begin{equation} 
    \begin{split}
    \log \E[(Z - Z_{\eps}(\bi^*)) \mathbbm{1}_{\mathcal{C}}] \leq -\frac{\eps}{p}  n \log n(1 + o(1))
        \end{split}
    \end{equation}
and so $Z - Z_{\eps}(\bi^*) \ll_{\Proba} 1 \leq  Z$. 

Conversely, when $\eta = -(\log n)^{-\frac{1}{5}}$, our job is to prove that for any $r < 1$, $\sup_{\bi \in \mathcal{S}_n^{p-1}} Z_r(\bi) \ll_{\Proba} Z$: by Corollary \ref{estimation}, this will imply that partial alignment is intractable. Here:

\begin{equation}
    \log \E[Z_r(\bi^*)\1_{\mathcal{C}}] \leq c_r (p-1)n (\log n)^{\frac{4}{5}}(1 + o(1))
\end{equation}
    for some $0 \leq c_r < 1$, by Proposition \ref{energy}. However, by Proposition \ref{hardquenched} and Lemma \ref{concentration}, 

    \begin{equation}
        \log Z \geq (p-1)|\eta|n\log n(1 - o_{\Proba}(1)) = (p-1)n (\log n)^{\frac{4}{5}}(1 -o_{\Proba}(1))
    \end{equation}

    so that { $Z_r(\bi^*) \ll_{\Proba} Z$.

    This is close to what we wanted, but not quite there: we still need to go from $Z_r(\bi^*) \ll_{\Proba} Z$ to the uniform bound $\sup_{\bi \in \mathcal{S}_n^{p-1}} Z_r(\bi) \ll_{\Proba} Z$. This requires another technical argument which we defer to Section \ref{sec}. With this extra argument, the case $\eta < 0$ is complete.}
    
    Finally, as for when $\eta > 0$, we still wish to show that there exists an \textbf{exact} estimator. For this, we will require the following extra lemma, whose proof is given in the appendix.

    \begin{lemma} \label{exact}
    Assume that $\eta > 0$. Then, there exists $\eps > 0$ such that, with high probability, for any $\bigma \in B(\bi^*, \eps)$ which is different from $\bi^*$, $V(\bigma) > 0$.
\end{lemma}
Since, for any $\bigma \notin B(\bi^*, \eps)$, $e^{-\beta V(\bigma)} \leq Z - Z_{\eps} \ll_{\Proba} 1$, this means that with high probability,
\begin{equation}
    \bi^* = \underset{\bi \in \mathcal{S}_n^{p-1}}{\arg \max} \Proba_{post}(\bi)
\end{equation}
concluding the proof of our main theorem.
\qed
\end{proof}

\section{The Erdős–Rényi model: proof of Theorem \ref{erthm}} \label{erproof}

In this section, we will provide a proof of Theorem \ref{erthm}. We define $\Gamma = (\Gamma^{(1)}, \ldots, \Gamma^{(p)}) = (\bi^*)^{-1}(G)$ (the observed graph), such that $\Proba_{post} = \Proba(\cdot | \Gamma)$. 

\subsection{A description of the system}

As in the Gaussian case, we will begin by evaluating the Hamiltonian of our system, and formulate a proof strategy based upon the behaviour of said Hamiltonian.

\begin{proposition} \label{hamiltonerdos}
Let $\bi = (\pi_1 = \Id , \pi_2, \ldots, \pi_p) \in \mathcal{S}_n^{p-1}$. For any subset $X \subseteq \llb 1, p \rrb$, we define $e_X = e_X(\bi)$ to be the number of edges $e \in E$ "of type $X$"; i.e. such that for any $1 \leq i \leq p$, $e \in \pi_i(\Gamma^{(i)})$ if and only if $i \in X$. (In this sense, $\sum_X e_X = \binom{n}{2}$.) 

With this setup, we may write

\begin{equation}
    \Proba_{post}(\bi^* = \bi) = \frac{1}{Z} e^{- \beta' \mathcal{H}(\bi)(1 + o(1))} \text{ with } \beta' = \log n; { \mathcal{H}(\bi) = -\sum_{X \neq \emptyset} e_X(\bi)}
\end{equation}
Here, $Z$ is a random variable which does not depend on $\bi$, and the error term $o(1)$ is bounded uniformly in $\bi$.

\end{proposition}

In other words, if we "glue" together all of our graphs using the correspondence $\bi$, the associated energy for this configuration is equal to the total number of edges in the combined graph (plus a smaller, negligible term). 

Also note that, contrary to the Gaussian case, we are now in a low-temperature configuration: we care very much about correctly matching our edges since we have so much fewer of them. 

\begin{proof}
Let $\bi \in \mathcal{S}_n^{p-1}$, and let $(A^{(1)}, \ldots, A^{(p)})$ be graphs over $\llb 1, n \rrb$. Then:

\begin{equation}
    \begin{split}
        \Proba(\bi^* = \bi, G^{(1)} = A^{(1)}, \ldots, G^{(p)} = A^{(p)}) & = \frac{1}{(n!)^{p-1}} \Proba( G^{(1)} = A^{(1)}, \ldots, G^{(p)} = A^{(p)} | \bi^* = \bi) \\
        & = \frac{1}{(n!)^{p-1}} \prod_{e \in E} \Proba(\forall i \in \llb 1, p \rrb, e \in G^{(i)} \Leftrightarrow e \in A^{(i)}) \\
    \end{split}
\end{equation}

since the variables $((\1_{e \in G^{(1)}}, \ldots, \1_{e \in G^{(p)}}))_{e \in E}$ are all independent, and independent from $\bi^*$. In fact, we know the laws of these variables explicitly, and so we may write:

\begin{equation}
    \Proba(\forall i \in \llb 1, p \rrb, e \in G^{(i)} \Leftrightarrow e \in A^{(i)}) = \begin{cases}
            1 - \frac{\lambda (1 - (1 - s)^{p})}{n} & \text{ if } \forall 1 \leq i \leq p, e \notin A^{(i)} \\
            \frac{\lambda}{n}s^k (1-s)^{p-k} & \text{ if } \# \{ 1 \leq i \leq p, e \in A^{(i)} \} = k > 0
        \end{cases}
\end{equation}

Consequently, if $B^{(1)}, \ldots, B^{(p)}$ are graphs over $\llb 1, n \rrb$:

\begin{equation}
\begin{split}
     \Proba & (\bi^* = \bi | \Gamma^{(1)} = B^{(1)}, \ldots, \Gamma^{(p)} = B^{(p)}) \\
     & = \frac{1}{Z'} \Proba(\bi^* = \bi, G^{(1)} = \pi_1(B^{(1)}), \ldots, G^{(p)} = \pi_p(B^{(p)})) \\
\end{split}
\end{equation}

meaning that

\begin{equation}
\begin{split}
     \Proba_{post}(\bi^* = \bi)& = \Proba(\bi^* = \bi | \Gamma) \\
     & = \frac{1}{Z'} \left( 1 - \frac{\lambda (1 - (1 - s)^{p})}{n} \right)^{e_{\emptyset}(\bi)} \prod_{X \subseteq \llb 1, p \rrb \setminus \emptyset} \left(\frac{\lambda s^{|X|} (1- s)^{p - |X|}}{n} \right)^{e_X(\bi)}
\end{split}
\end{equation}

Since $\sum_{X \subseteq \llb 1, p \rrb }e_X(\bi) = \binom{n}{2}$ is independent from $\bi$, we may therefore rewrite:

\begin{equation}
\begin{split}
     \Proba_{post} (\bi^* = \bi) & = \frac{1}{Z'} \prod_{X \subseteq \llb 1, p \rrb \setminus \emptyset} \left(\frac{\lambda s^{|X|} (1- s)^{p - |X|}}{n} \right)^{e_X(\bi)} \left( 1 - \frac{\lambda (1 - (1 - s)^{p})}{n} \right)^{-e_{X}(\bi)} \\
    & = \frac{1}{Z} \exp \left( - \log n\sum_{X \neq \emptyset} e_X(\bi) (1 - \frac{1}{\log n} \log \left( \frac{1 - \frac{\lambda (1 - (1 - s)^{p})}{n} }{\lambda s^{|X|} (1- s)^{p - |X|}}\right)) \right)\\
    & = \frac{1}{Z} \exp \left( - \mathcal{H}(\bi) \log n (1 + o(1))) \right)
\end{split}
\end{equation}

which is what we wanted to show.
\qed
\end{proof}

\subsection{Symmetries of $\mc{H}$}

{ Recall that our aim is to apply Lemma \ref{invariance} in order to show Theorem \ref{erthm}. To this end, we fix a ball $B = B(\bi_0, r) \subseteq \mc{S}_n^{p-1}$: we wish to construct a subset $\text{Neg} \subseteq \mc{S}_n^{p-1}$ and a function $F: B \setminus \text{Neg} \to \mc{P}(\mc{S}_{n}^{p-1})$ such that

\begin{enumerate}
    \item for some $c > 0$, for any $\bi \in B \setminus \text{Neg}$, $\Proba_{post}(\bi) \leq \frac{1}{(cn)!}\Proba_{post}(F(\bi))$;
    \item if $\bi' \in \mc{S}_n^{p-1}$, the set $F^{-1}(\bi') \defeq \{ \bi \in B \setminus \text{Neg}: \bi' \in F(\bi) \}$ has size at most $\exp O_\Proba(n)$;
    \item $\Proba_{post}(\text{Neg}) = o_\Proba(1)$.
\end{enumerate}
This will guarantee that the prerequisites of Lemma \ref{invariance} are verified, and thus prove Theorem \ref{erthm}.

In this subsection, we will describe the symmetries which we intend to make use of, in order to construct our function $F$. In the following subsection, we will define our set $\text{Neg}$ verifying (iii), in such a way that, if $\bi \notin \text{Neg}$, there are many of these symmetries which we may exploit in order to build $F$. Then, in our final subsection, we will explicitly construct $F$, and prove that conditions (i) and (ii) are verified - showing the theorem.
}

Thus, let us begin by describing the symmetries which we are interested in.

\begin{proposition} \label{hamiltonianinvariant}
    Let $\bi \in \mathcal{S}_n^{p-1}$; set 
    $\mc{U}^{(1)} = \mc{U}^{(1)}(\bi) = \ds \bigcup_{i=2}^p \pi_i(\Gamma^{(i)})$.

Then, if $\sigma \in \mathcal{S}_n$ is an automorphism of $\Gamma^{(1)} \cap \mc{U}^{(1)}$, setting $\bigma = (\Id, \sigma, \ldots, \sigma) \in \{ \Id \} \times \mathcal{S}_n^{p-1}$,

    \begin{equation}
        \mathcal{H}(\bigma \bi) \defeq \mathcal{H}(( \pi_1,  \sigma \circ \pi_2, \ldots, \sigma \circ \pi_p)) \leq \mathcal{H}(\bi).
    \end{equation}
\end{proposition}
{
The reason why such symmetries interest us is that $\mathcal{H}(\bigma \bi) \leq \mathcal{H}( \bi)$ corresponds (roughly) to having $\Proba_{post}(\bigma \bi) \geq \Proba_{post}(\bi)$; thus, if our function $F$ is of the form 

\begin{equation}
    \bi \mapsto \{ \bigma_1 \bi, \ldots, \bigma_N \bi \}
\end{equation}

for some large $N$, then condition (i) from above will automatically be verified. 
}
\begin{remark}
    There is of course nothing special about $\Gamma^{(1)}$: we could have defined $\mc{U}^{(i)}(\pi)$ similarly for any $2 \leq i \leq p$, and the same proposition would hold for any automorphisms of $\Gamma^{(i)} \cap \mc{U}^{(i)}$.
\end{remark}

\begin{proof}
    Let $\sigma$ be an automorphism of $\Gamma^{(1)} \cap \mc{U}^{(1)}(\bi)$, and denote $\bigma = (\Id, \sigma, \ldots, \sigma) \in \{ \Id\} \times \mathcal{S}_n^{p-1}$. We will note a few identities relating the $e_X(\bigma \bi)$ to the $e_X(\bi)$.

First of all,

\begin{equation}
    \sum_{X \neq \emptyset, \{ 1 \} } e_X(\bigma \bi) = e(\sigma(\mc{U}^{(1)})) = e(\mc{U}^{(1)}) = \sum_{X \neq \emptyset, \{ 1 \} } e_X(\bi)
\end{equation}
and
\begin{equation}
    \sum_{1 \in X \subseteq \llb 1, p \rrb} e_X(\bigma \bi) = e(\Gamma^{(1)}) = \sum_{1 \in X \subseteq \llb 1, p \rrb} e_X(\bi)
\end{equation}

(These identities are true for any $\sigma \in \mathcal{S}_n$.) Furthermore, since $\sigma$ is an automorphism of $\Gamma^{(1)} \cap \mc{U}^{(1)}$,

\begin{equation}
\begin{split}
       \sum_{\underset{X \neq \{ 1 \}}{1 \in X \subseteq \llb 1, p \rrb}}e_X(\bigma \bi) & = e(\Gamma^{(1)} \cap \sigma(\mc{U}^{(1)})) \\
       & \geq e((\Gamma^{(1)} \cap \mc{U}^{(1)}) \cap \sigma(\Gamma^{(1)} \cap \mc{U}^{(1)})) \\
       & = e(\Gamma^{(1)} \cap \mc{U}^{(1)}) \\
       & = \sum_{\underset{X \neq \{ 1 \}}{1 \in X \subseteq \llb 1, p \rrb}}e_X(\bi)
\end{split}
\end{equation}

Putting these three identities together, we obtain:

\begin{equation}
    \begin{split}
        \sum_{X \neq \emptyset} e_X(\bigma \bi) & = \left( \sum_{X \neq \emptyset, \{ 1\}} e_X(\bigma \bi) \right) + e_{\{ 1 \}}(\bigma \bi)\\
        & = \sum_{X \neq \emptyset, \{ 1\}} e_X(\bigma \bi) + \sum_{1 \in X \subseteq \llb 1, p \rrb} e_X(\bigma \bi) - \sum_{\underset{X \neq \{ 1 \}}{1 \in X \subseteq \llb 1, p \rrb}} e_X(\bigma \bi) \\
        & \leq \sum_{X \neq \emptyset, \{ 1\}} e_X(\bi) + \sum_{1 \in X \subseteq \llb 1, p \rrb} e_X(\bi) - \sum_{\underset{X \neq \{ 1 \}}{1 \in X \subseteq \llb 1, p \rrb}} e_X(\bi) \\
        &  = \sum_{X \neq \emptyset} e_X(\bi) 
    \end{split}
\end{equation}
concluding the proof. \qed
\end{proof}

\subsection{Building many automorphisms}

We are now tasked with finding such automorphisms $(\bigma_k)_k$ of $A(\bi) \defeq \Gamma^{(1)} \cap \mc{U}^{(1)}(\bi)$. For the true alignment $\bi^*$, given that the intersection graph $\Gamma^{(1)} \cap \mc{U}^{(1)}(\bi^*)$ is an Erdős–Rényi graph, this is possible when that graph is subcritical. Specifically:

\begin{proposition} \label{whatisneg}
    Let $G$ be a graph over $\llb 1, n \rrb$, and let $c > 0, \delta > 0$. We will say that $G$ verifies the property $(P_{c, \delta})$ if, for any $X \subseteq [n]$ of size at least $\delta n$ with no outgoing edges, there are at least $(c n)!$ automorphisms of $G$ which leave $X^c = (\llb 1,n \rrb \setminus X)$ fixed, and which have at most $\delta^2 n$ fixed points within $X$.

Now, assume that $A$ is an Erdős–Rényi graph with size $n$ and parameter $\rho < 1$, and let $\delta > 0$. Then, there exists $c = c(\rho, \delta)$ such that, with high probability, $A$ verifies the property $(P_{c, \delta})$.
\end{proposition}

Essentially, a graph verifying the property $(P_{c, \delta})$ is one such that any reasonably large subgraph has a lot of symmetries: this is naturally the case for a subcritical Erdős–Rényi graph, which is locally a forest whose trees can be swapped around.

\begin{proof}
    Let $A$ be as in the proposition. Fix an ordering $(\mathcal{T}_1, \mathcal{T}_2, \ldots)$ of all isomorphism classes of trees. For any $i \geq 1$, we set $n_{i}$ to be the number of connected components of $A$ which are isomorphic to $\mathcal{T}_i$. Then, for any $i$, there exists a constant $c_i$ with 

\begin{equation}
    \frac{n_i}{n} \toPr c_{i}
\end{equation}

and $\ds\sum_{i \geq 1} |\mathcal{T}_i|c_i = 1 $. (This is shown in \cite{bollobas}.)

Thus, let $\delta > 0$. If $K > 0$ is large enough, with high probability,

\begin{equation}
    \sum_{\underset{|\mathcal{T}_i| \leq K}{i \geq 1}} |\mathcal{T}_i| \frac{n_i}{n} \geq 1 - \frac{\delta^2}{2}
\end{equation}

As a result, any $X \subseteq \llb 1, n \rrb$ with size at least $\delta n$ and with no outgoing edges, contains a subset $X' \subseteq X$ of size at least $(\delta - \frac{\delta^2}{2}) n$ which is a disjoint union of connected components of size at most $K$.

Call $m_i$ the number of connected components of $X'$ isomorphic to $\mathcal{T}_i$, for any $i$: then, we may assume that $m_i \neq 1$ for any $i \geq 1$, even if that means decreasing the size of $X'$ to be at least $(\delta - \delta^2) n$. Consequently, we may build automorphisms of $X'$ by permuting its connected components which are isomorphic. The number of automorphisms with no fixed point which we can build in this way is at least equal to 

\begin{equation}
    \prod_{\underset{|\mathcal{T}_i| \leq K}{i \geq 1}} \frac{m_i!}{3}
\end{equation}

since, for any $k \neq 1$, there are at least $\frac{k!}{3}$ elements of $\mathcal{S}_k$ with no fixed points. Since $K$ is independent from $n$ and $\sum_i m_i \geq \ds\frac{(\delta-\delta^2) n}{K}$, this provides us with at least

\begin{equation}
    \exp\left({\frac{(\delta-\delta^2)}{K} n \log n + O(n)} \right) \geq (cn)!
\end{equation}

automorphisms of $X'$, for large enough $n$ (and some $c > 0$ independent from $n$). 

Finally, we may extend these automorphisms to be $\Id$ on the complement of $X'$, proving the proposition. \qed
\end{proof}

The previous proposition concerns the structure of $A(\bi^*)$. However, the measure $\Proba_{post}$ puts a large weight on permutations which are "hard to distinguish" from $\bi^*$; such permutations $\bi$ will give us a graph $A(\bi)$ with a similar structure. This is captured in the following corollary.

\begin{corollary} \label{corneg}
    Let $\delta > 0$. If $\lambda s (1 - (1 - s)^{p-1}) < 1$, the graph $A(\bi^*) = \Gamma^{(1)} \cap \mc{U}^{(1)}(\bi^*)$ is subcritical: there exists $c > 0$ such that $A$ verifies the property $(P_{c, \delta})$ defined above with high probability. Otherwise put, the set
    \begin{equation}
        \Neg = \{ \bi \in \mathcal{S}_n^{p-1} | \Gamma^{(1)} \cap \mc{U}^{(1)}(\bi) \text{ does not verify }(P_{c, \delta}) \}
    \end{equation}
    verifies $\Proba(\pi^* \in \Neg) = o(1)$. Consequently, $\Proba_{post}(\pi^* \in \Neg) = o_{\Proba}(1)$.
\end{corollary}

\subsection{Applying the lemma}

We now have most of the ingredients necessary to apply Lemma \ref{invariance} and thus prove Theorem \ref{erthm}.

\begin{proof}[Proof of Theorem \ref{erthm}]
    Let $\bi_0 \in \mathcal{S}_n^{p-1}$ and $\eps > 0$; we will build $F: B(\bi_0, 1- \eps) \setminus \Neg \to \mc{P}(\mathcal{S}_n^{p-1})$ verifying the conditions of Lemma \ref{invariance}.

    Thus, let $\bi \in B(\bi_0, 1-\eps) \setminus \Neg$. We consider the set $X(\bi) = \{ u \in \llb 1, n \rrb, \forall i \in \llb 1, n \rrb,  \pi_i(u) = (\bi_0)_i(u) \}$; by construction, $|X| \geq \eps n$. We then consider the set $\overline{X}(\bi)$, which is the minimal subset of $\llb 1, n \rrb$ containing $X$ such that no edges of $A(\bi)$ link $\overline{X}$ and $\overline{X}^c$. Since $\bi \notin \Neg$, by Proposition \ref{whatisneg}, we may construct $(c n)!$ automorphisms $(\sigma_1, \ldots, \sigma_N)$ of the graph $A$ which leave $\overline{X}^c$ fixed, and which have at most $\eps^2 n$ fixed points within $\overline{X}$, for some $c > 0$ independent of $n$. The function $F$ is defined as mapping $\bi$ to the set $\{ \bigma_i \bi, 1 \leq i \leq N \}$ (where - as defined earlier - $\bigma_i = (\Id, \sigma_i, \ldots, \sigma_i)$).

    Does $F$ satisfy the conditions of Lemma \ref{invariance}? Condition (i) is verified by construction: for any $\bi \in B(\bi_0, 1-\eps)$,
    \begin{equation}
        \Proba_{post}(\bi) \leq \frac{e^{o(n \log n)}}{(cn)!} \Proba_{post}(F(\bi)) 
    \end{equation}
    since $\mathcal{H}(\bigma_i \bi) \leq \mathcal{H}(\bi)$ for any $1 \leq i \leq N$. Condition (ii), however, is less obvious. We will prove that it holds with $K_n(r) = \exp O_{\Proba}(n)$. 

    Let $\bi' \in \mathcal{S}_n^{p-1}$. We may partition $F^{-1}(\bi') = \{ \bi: \bi' \in F(\bi) \}$ as 
    \begin{equation}
    \begin{split}
        F^{-1}(\bi')&  = \bigsqcup_{\underset{S \text{ subgraph of } (\Gamma^{(1)})|_Z}{Y \subseteq Z \subseteq \llb 1, n \rrb}} \{ \bi \in F^{-1}(\bi'): X(\bi) = Y, \overline{X}(\bi) = Z, (\Gamma^{(1)}\cap \mc{U}^{(1)}(\bi))|_Z = S \} \\
        & \defeq  \bigsqcup_{\underset{S \text{ subgraph of } (\Gamma^{(1)})|_Z}{Y \subseteq Z \subseteq \llb 1, n \rrb}} W(Y, Z, S)
        \end{split}
    \end{equation}
    Then, there are only $\exp O_\Proba(n)$ possible values for $(Y, Z, S)$, as $|\Gamma^{(1)}| = O_{\Proba}(n)$. We therefore just need to show that $|W(Y, Z, S)| \leq \exp O_\Proba(n)$ for any triplet $(Y, Z, S)$. To this end, let $\bi_1, \bi_2 \in W(Y, Z, S)$: by construction, $\bi_2 \bi_1^{-1}$ is of the form $(\Id, \sigma, \ldots, \sigma)$ for some $\sigma \in \mathcal{S}_n$. Furthermore, given such a $\sigma$, we know that $\sigma|_Y = \Id$, and $\sigma|_Z$ must be an automorphism of $S$. This allows us to use the following lemma, whose proof is deferred to the appendix.
    \begin{lemma} \label{auttrees}
        Let $H$ be a graph over some set $V$, with degree distribution $(d_v)_{v \in V}$. Then, $H$ has at most $e^{\frac{|V|}{e}}\ds\prod_{v \in V}\left(d_v!\right)$ automorphisms which send each connected component of $H$ to itself.
    \end{lemma}
    Indeed, by construction of $\overline{X}(\bi_1) = Z$, any connected component of $S$ intersects $Y$, and so $\sigma$ must send it to itself. This gives us at most
        $e^{\frac{n}{e}}\prod_{i=1}^n (d_i!) $
    possible values for $\sigma$, showing that $(ii)$ is verified as 

    \begin{equation}
        \log \prod_{i=1}^n (d_i!) \leq \sum_{i=1}^n d_i \log d_i = O_{\Proba}(n \E[d_i \log d_i ]) = O_{\Proba}(n).
    \end{equation}
    given that $\E[d_i \log d_i ] = O(1)$. Applying Lemma \ref{invariance} thus proves Theorem \ref{erthm}. \qed
\end{proof}

\appendix

\section{Some results on quadratic forms of normal vectors}

In this section, we list a few results on the behaviour of quadratic forms of normal vectors, for use in the following sections.

We begin with a general lemma on cumulants of quadratic forms of normal vectors. 

\begin{lemma} \label{cumcalc}
    Let $M$ be a symmetric matrix: then, if $A \sim \mathcal{N}(0, \Sigma)$, and $Q = A^T M A$ is a quadratic form of $A$, the log-moment generating function of $Q$ exists for small enough $t$, and can be written as

\begin{equation}
    \log \E[e^{t Q}] = - \frac{1}{2} \log \det(\Id - 2 t \Sigma M) =  \sum_{k \geq 1} 2^{k-1}\Tr((\Sigma M)^k) \frac{t^k}{k}
\end{equation}
\end{lemma}

The following proof is a standard calculation, often used to prove the Hanson-Wright inequality (see \cite{hr} for instance).

\begin{proof}
    Replacing $A$ by $\Sigma^{-\frac{1}{2}} A$, it is sufficient for us to handle the case where $\Sigma = \Id$.

    We may diagonalise $M$ to obtain the expression $M = P^T D P$, where $P$ is an orthogonal matrix. Then, we may write

    \begin{equation}
        Q = (PA)^T D (PA) \defeq Y^T D Y
    \end{equation}
    However, since the standard Gaussian distribution is invariant by the action of the orthogonal group $O(n)$, $Y$ is also a (standard, centred) normal vector. As such, setting $Y = (Y_1, \ldots, Y_n)$ and denoting by $(\lambda_1, \ldots, \lambda_n)$ the eigenvalues of $M$,
\begin{equation}
\begin{split}
    \E[e^{t Q}] & = \prod_{i=1}^n \E[e^{t\lambda_i Y_i^2}] \\
                & = \prod_{i=1}^n \frac{1}{\sqrt{1 - 2 t \lambda_i}} \\
                & = (\det(\Id - 2 t M))^{-\frac{1}{2}}
\end{split} 
\end{equation}
for small enough $t$, showing the first identity.

For the second one, note that if $t$ is small enough, we may write $\Id - 2 t M = e^L$ for some matrix $L = - \dsum_{k \geq 1} 2^k M^k \frac{t^k}{k}$. Thus,

\begin{equation}
    \log \det (\Id - 2 t M)  = \Tr L  = -  \dsum_{k \geq 1} 2^k \Tr(M^k) \frac{t^k}{k}
\end{equation}
concluding the proof.\qed
\end{proof}

In the same vein as this, we may obtain tail bounds on such random variables. Again, the following result is fairly standard: a proof can be found in \cite{concentration} for instance.

\begin{lemma} \label{weakmass}
    Let $X$ be a centred normal vector with covariance matrix $\Sigma$, $S_N$ a symmetric 
    matrix with size $N \to + \infty$. Then,

    \begin{equation}
    \begin{split}
        \log \Proba(X^T S_N X  - \E[X^T S_N X] > C_N) & \leq -\frac{C_N^2}{4(\Tr((\Sigma S_N)^2) + C_N \|\Sigma^{\frac{1}{2}} S_N \Sigma^{\frac{1}{2}}\|_{op})} \\
        & = -\frac{C_N^2}{2 \Var(X^T S_N X) + C_N \|\Sigma^{\frac{1}{2}} S_N \Sigma^{\frac{1}{2}}\|_{op})}
        \end{split}
    \end{equation}
\end{lemma}

Finally, we will state a significantly sharper tail bound which we may apply if the quantity $C_N$ is not too large.

\begin{proposition}
\label{mass2}
        Let $X$ be a centred normal vector with covariance matrix $\Sigma$, $S_N$ a symmetric 
    matrix with size $N \to + \infty$, verifying {$\|\Sigma^{\frac{1}{2}}S_N \Sigma^{\frac{1}{2}} \|_{op} \leq K$ }independent of $N$. Then, if $C_N > 0$ with $1\ll C_N  \ll \Tr((\Sigma S_N)^2)$,

    \begin{equation}
         \Proba(X^T S_N X  - \E[X^T S_N X] > C_N) \underset{N \to + \infty}{\sim} \frac{\sqrt{\Tr((\Sigma S_N)^2)}}{\sqrt{\pi} C_N}\exp \left(- \frac{C_N^2}{4 \Tr((\Sigma S_N)^2)} + O(\frac{C_N^3}{\Tr((\Sigma S_N)^2)^2}) \right)
    \end{equation}

    and a symmetric result holds for $\Proba(X^T S_N X  - \E[X^T S_N X] < - C_N)$.
\end{proposition}

\begin{proof}

    Again, we may write

    \begin{equation}
        Q_N = X^T S_N X - \E[X^T S_N X] = \sum_{i=1}^N \lambda_i (Y_i^2-1)
    \end{equation}
    as in the previous proof.
    
    This time, we will apply the following theorem from \cite[Eq. (2.11)]{hrsharp}:

    \begin{theorem}
        Let $\xi_1, \ldots, \xi_N$ be independent centred random variables such that, for some $A > 0$, for any $i, k \geq 0$,

        \begin{equation}
            \E[|\xi_i|^k] \leq A^k k! 
        \end{equation}

        Then, setting $\Phi(x) = \frac{1}{\sqrt{2\pi}} \int_{-\infty}^x e^{-\frac{u^2}{2}} du$, and $s^2 = \sum_{i} \E[\xi_i^2]$, if $1 \ll x \ll s$,

        \begin{equation}
            \Proba\left( \dfrac{1}{s} \sum_{i=1}^N \xi_i > x \right) \underset{N \to + \infty}{\sim} [1 - \Phi(x)]e^{O(\frac{x^3}{s})} \underset{N \to + \infty}{\sim} \frac{1}{\sqrt{2\pi}x} e^{-\frac{x^2}{2} + O(\frac{x^3}{s})}
        \end{equation}
    \end{theorem}

    In this case, $s^2 = 2\ds\sum_{i=1}^n \lambda_i^2 = 2\Tr((\Sigma S_N)^2)$ and so $x^2 = \dfrac{C_N^2}{2 \Tr((\Sigma S_N)^2)}$; Proposition \ref{mass2} follows immediately.
    \qed
\end{proof}
\section{Proof of Proposition \ref{cumulants}} \label{quad2}

We now turn to proving Proposition \ref{cumulants}. Here, for any $\bigma \in \mathcal{S}_n^{p-1}$, we may write $V_{\mathrm{off}}(\bigma) = -G^T M_{\bigma} G$, as a quadratic form over $\R^E \otimes \R^{p}$, where

\begin{equation} \label{defmsigreal}
    (M_{\bigma})_{(e, i), (e', j)} = \1_{e'=\sigma_{ij}(e)} \1_{e \neq e'} \1_{i \neq j}
\end{equation}

for any $e, e' \in E, 1 \leq i , j \leq p$; and, as described in (\ref{covdef}), the covariance matrix of $G$ is 

\begin{equation} \label{covmatrix}
    \Sigma = (1 - \rho) I_N + \rho ( I_{\binom{n}{2}} \otimes J_p) \defeq (1 - \rho) I_N + \rho \tilde{J}
\end{equation}

where $J_p$ is the $p \times p$ all-ones matrix. For the sake of convenience, we will set $\mathcal{M}_{\bigma} = \Sigma^{\frac{1}{2}} M_{\bigma}\Sigma^{\frac{1}{2}}$ for the rest of this proof.

We will also note that, for symmetric positive $U$ and symmetric $V$, 

\begin{equation} \label{traces}
    \Tr UV \leq \|V\|_{op} \Tr U.
\end{equation}
This will allow us to bound the traces of our matrices in order to properly control our error terms. { Explicitly, since $\|M_{\bigma}\|_{op} \leq p-1$ and $\|\Sigma\|_{op} = 1 + o(1)$, we may use the fact that $ \mathcal{M}_{\bigma}^2$ is positive, so that for $k \geq 2$ and large enough $n$ 

\begin{equation} \label{crude}
\begin{split}
    \Tr(\mathcal{M}_{\bigma}^k) &\leq \Tr(\mathcal{M}_{\bigma}^2) \|\mathcal{M}_{\bigma}^{k-2}\|_{op} \\
    & \leq p^{k-2}\Tr(\mathcal{M}_{\bigma}^2)
    \end{split}
\end{equation}

This allows us, by applying Lemma \ref{cumcalc}, to obtain that:

\begin{equation}
    \begin{split}
        \log \E[e^{-\beta V_{\mathrm{off}}(\bigma)}] & = \sum_{k \geq 1} 2^{k-1} \Tr((\Sigma M_{\bigma})^k) \frac{\beta^k}{k} \\
        & = \beta \Tr(\mathcal{M}_{\bigma}) + \beta^2 \Tr (\mathcal{M}_{\bigma}^2)(1 + o(1)).
    \end{split}
\end{equation}
}
We now simply need to evaluate some traces.

To begin with, by definition, $\Sigma$ and $M_{\bigma}$ can be written as block matrices with ($|E|$ blocks $ \times $ $ |E|$ blocks), where each block is of size $p \times p$. The matrix $\Sigma$ is block diagonal, while the diagonal blocks of $M_{\bigma}$ are all zero matrices. Consequently, $\Tr(\Sigma M_{\bigma}) = 0$.

We now move on to $\Tr(\mathcal{M}_{\bigma}^2)$. Given that $\|\Sigma - \Id\|_{op} = O(\rho)$,
{
\begin{equation}
\begin{split}
     \Tr(\mathcal{M}_{\bigma}^2) & = \Tr(\Sigma M_{\bigma}\Sigma M_{\bigma}) \\
     & = \Tr(M_{\bigma}^2) + 2\Tr((\Sigma - \Id) M_{\bigma}^2) + \Tr((\Sigma - \Id) M_{\bigma}(\Sigma - \Id) M_{\bigma}) \\
\end{split}
\end{equation}

so that

\begin{equation}
    \begin{split}
        \Tr(\mathcal{M}_{\bigma}^2) - \Tr(M_{\bigma}^2) & \leq 2 \|\Sigma - \Id\|_{op} \Tr(M_{\bigma}^2) +\|\Sigma - \Id\|_{op} \Tr((\Sigma - \Id)M_{\bigma}^2) \\
        & = \Tr(M_{\bigma}^2)(1 + O(\rho)).
    \end{split}
\end{equation}

by (\ref{traces}).}

We may now compute:
\begin{equation}
\begin{split}
    \Tr(M_{\bigma}^2) & = \sum_{e, e' \in E} \sum_{1 \leq i, j \leq p} (M_{\bigma})_{(e, i), (e', j)}^2 \\
    & = \# \{ e \neq e' \in E, 1 \leq i \neq j \leq p, e' = \sigma_{ij}(e)  \} \\
    & = \sum_{ 1 \leq i \neq j \leq p} \left(\binom{n}{2} - \# \{ e \in E, \sigma_{ij}(e) = e \}\right)\\
    & \defeq \sum_{1 \leq i \neq j \leq p} \left(\binom{n}{2} - D_{ij}(\bigma) \right)
    \end{split}
\end{equation}

Thus:

\begin{equation}
    \log \E[e^{-\beta V_{\mathrm{off}}(\bigma)}]  = \beta^2 \sum_{1 \leq i \neq j \leq p} \left(\binom{n}{2} - D_{ij}(\bigma) \right) + O(\beta^3 n^2)
\end{equation}

Note that $\beta^3 n^2 = o(n)$; thus, to conclude the proof, we just need to show the following lemma.

\begin{lemma} \label{Dij}
    For $\bigma \in \mathcal{S}_n^{p-1}$,
    \begin{equation}
    D_{ij}(\bigma) = \binom{d_{ij}(\bigma)}{2} + R_{\bigma}
\end{equation}
where $0 \leq R_{\bigma} \leq \frac{n - d_{ij}(\bigma)}{2}$.
\end{lemma}
\begin{proof}
    If $e \in E$, writing $e = \{u, v \}$, we may see that $\sigma_{ij}(e) = e$ if and only if

    \begin{equation}
    \begin{cases}
            \sigma_{ij}(u) = u \\
            \sigma_{ij}(v) = v
        \end{cases} \text{ or       }
    \begin{cases}
        \sigma_{ij}(u) = v \\
        \sigma_{ij}(v) = u
    \end{cases}.
\end{equation}

There are $\ds \binom{d_{ij}(\bigma)}{2}$ pairs verifying the first condition. The pairs verifying the second condition are the orbits of length $2$ for the action of $\sigma_{ij}$ over $\llb 1, n \rrb$: there are at most $\dfrac{n - d_{ij}}{2}$ of these.  \qed
\end{proof}

We have thus proven Proposition \ref{cumulants}. \qed 

Before moving on, we will perform a final computation which we will need later on. Namely, if $\bigma, \bigma' \in \mathcal{S}_n^{p-1}$:

\begin{equation}
\begin{split}
    \Tr(M_{\bigma} M_{\bigma'}) & = \sum_{e, e' \in E} \sum_{1 \leq i, j \leq p} (M_{\bigma})_{(e, i), (e', j)} (M_{\bigma'})_{(e, i), (e', j)} \\
    & = \# \{ e \neq e' \in E, 1 \leq i \neq j \leq p, e' = \sigma_{ij}(e) = \sigma'_{ij}(e) \} \\
    & = \sum_{ 1 \leq i \neq j \leq p}\# \{ e \in E, \sigma_{ij}(e) = \sigma'_{ij}(e) \neq e \}\\
    & \leq \sum_{ 1 \leq i \neq j \leq p}\# \{ e \in E, \sigma_{ij}(e) = \sigma'_{ij}(e) \}\\
    & \defeq \sum_{1 \leq i \neq j \leq p} C_{ij}(\bigma, \bigma')
    \end{split}
\end{equation}

As for $(D_{ij}(\bigma))$, we may show that, setting $c_{ij}(\bigma, \bigma') = \{u \in \llb 1, n \rrb, \sigma_{ij}(u) = \sigma'_{ij}(u) \}$,

\begin{equation}
    C_{ij}(\bigma, \bigma') = \binom{c_{ij}(\bigma, \bigma')}{2} + O(n)
\end{equation}

and so

\begin{equation} \label{C17}
    \begin{split}
        \Tr((\mathcal{M}_{\bigma} + \mathcal{M}_{\bigma'})^2) & = \Tr((M_{\bigma} + M_{\bigma'})^2) + O(\rho n^2) \\
        & = \Tr(M_{\bigma}^2) + \Tr(M_{\bigma'}^2) + 2 \Tr(M_{\bigma} M_{\bigma'}) \\
        & = \sum_{1 \leq i \neq j \leq p}\left( \binom{n}{2} - \binom{d_{ij}(\bigma)}{2} - \binom{d_{ij}(\bigma')}{2} + 2 \binom{c_{ij}(\bigma, \bigma')}{2} \right) + O(\rho n^2).
    \end{split}
\end{equation}
{
\section{Proof of Proposition \ref{weaklemma}} \label{quad}

Recall that the value of $V_{\mathrm{diag}}(\bigma)$ is equal to 
\begin{equation}
\begin{split}
    V_{\mathrm{diag}}(\bigma) & = -\sum_{\underset{\sigma_{ij}(e) \neq e}{e \in E, 1 \leq i \neq j \leq p}}G_e^{(i)} G_e^{(j)} \\
    & = \sum_{1 \leq i \neq j \leq p} \left[-\sum_{e \in E} G_e^{(i)} G_e^{(j)} +  \sum_{\underset{u, v \in \Fix \sigma_{ij}}{e = \{u, v \} \in E}}G_e^{(i)}  G_e^{(j)} + \sum_{\underset{\sigma_{ij}(u) = v, \sigma_{ij}(v) = u}{e = \{u, v \} \in E}}G_e^{(i)} G_e^{(j)}\right] \\
    & \defeq \sum_{1 \leq i \neq j \leq p} \left[ V_{ij}^{(1)} - V_{ij}^{(2)}(\Fix \sigma_{ij}) - V_{ij}^{(3)}(\bigma) \right]
    \end{split}
\end{equation}

Since this is a quadratic form of the vector $(G_{e}^{(i)})$, we may apply Lemma \ref{weakmass} to bound the probability that any of these three terms ever becomes too large. First,

\begin{equation}
    \begin{split}
        \Proba & (\exists 1 \leq i \neq j \leq p, \exists Y \subseteq \llb 1, n \rrb, | V_{ij}^{(2)}(Y) - \E[ V_{ij}^{(2)}(Y)]| > (\log n)^{\frac{1}{4}}n^{\frac{3}{2}}) \\
        & \leq \sum_{1 \leq i \neq j \leq p} \sum_{ Y \subseteq \llb 1, n \rrb} \Proba(| V_{ij}^{(2)}(Y) - \E[ V_{ij}^{(2)}(Y)]| > (\log n)^{\frac{1}{4}}n^{\frac{3}{2}}) \\
        & \leq \sum_{1 \leq i \neq j \leq p} \sum_{ Y \subseteq \llb 1, n \rrb} \exp\left(- \frac{n^3 (\log n)^{\frac{1}{2}}}{O(n^2) + O((\log n)^{\frac{1}{4}}n^{\frac{3}{2}})}\right) \\
        & = \exp(- c n \sqrt{\log n} + o(n \sqrt{\log n})) = o(1)
    \end{split}
\end{equation}
for some $c > 0$, since $\Var V_{ij}^{(2)}(Y) = O(n^2) $ (uniformly in $Y$) and the operator-norm term is $O(1)$ (again uniformly in $Y$).

Noting that $V_{ij}^{(1)} = V_{ij}^{(2)}(\llb 1, n \rrb)$, this also shows that 

\begin{equation}
\begin{split}
    \Proba & (\exists 1 \leq i \neq j \leq p, | V_{ij}^{(1)} - \E[ V_{ij}^{(1)}]| > (\log n)^{\frac{1}{4}}n^{\frac{3}{2}}) \\
    & \leq \exp(- c n \sqrt{\log n} + o(n \sqrt{\log n})) = o(1).
    \end{split}
\end{equation}

Finally, 

\begin{equation}
    \begin{split}
        \Proba & (\exists 1 \leq i \neq j \leq p, \exists \bigma \in \mc{S}_n^{p-1},  | V_{ij}^{(3)}(\bigma) - \E[ V_{ij}^{(3)}(\bigma)]| > (\log n)^{\frac{1}{4}}n^{\frac{3}{2}}) \\
        & \leq \sum_{1 \leq i \neq j \leq p} \sum_{ \bigma \in \mc{S}_n^{p-1}} \Proba(| V_{ij}^{(3)}(\bigma) - \E[ V_{ij}^{(3)}(\bigma)]| > (\log n)^{\frac{1}{4}}n^{\frac{3}{2}}) \\
        & \leq \sum_{1 \leq i \neq j \leq p} \sum_{ \bigma \in \mc{S}_n^{p-1}}\exp\left(- \frac{n^3 (\log n)^{\frac{1}{2}}}{O(n) + O((\log n)^{\frac{1}{4}}n^{\frac{3}{2}})}\right) \\
        & = \exp(- c n^{\frac{3}{2}}(\log n)^{\frac{1}{4}} + o(n^{\frac{3}{2}}(\log n)^{\frac{1}{4}})) = o(1).
    \end{split}
\end{equation}

as, this time, $\Var V_{ij}^{(3)}(\bigma) = O(n)$ (uniformly in $\bigma$). This proves the first part of the proposition.

The second part follows, simply by noting that, if $D_{ij}(\bigma) = \# \{ e \in E, \sigma_{ij}(e) = e \}$ for $1 \leq i \neq j \leq p$,

\begin{equation}
\begin{split}
    \E[V_{\mathrm{diag}}(\bigma)] & = -  \sum_{1 \leq i \neq j \leq p} \sum_{\underset{\sigma_{ij}(e) \neq e}{e \in E}} \E[G_e^{(i)} G_e^{(j)}] \\
    & = - \rho \sum_{1 \leq i \neq j \leq p} \left(\binom{n}{2} - D_{ij}(\bigma)\right) \\
    & = - \rho \sum_{1 \leq i \neq j \leq p} \left( \binom{n}{2} -\binom{d_{ij}(\bigma)}{2}\right) + O(\rho n)
    \end{split}
\end{equation}

by Lemma \ref{Dij}.
}
{
\section{Proof of Proposition \ref{energy}}
If $r\geq 0$,

\begin{equation}
\begin{split}
    \E[Z_r(\bi^*) \1_{\mathcal{C}}] & = \sum_{\bigma \in B(\boldsymbol{\pi}^*,r) } \E[e^{-\beta V(\bigma)}\1_{\mathcal{C}}] \\
    & \leq \sum_{\bigma \in B(\boldsymbol{\pi}^*,r)} \E[e^{- \beta(V_{\mathrm{off}}(\bigma) + \E[V_{\mathrm{diag}}(\bigma)] + O((\log n)^{\frac{1}{4}} n^{\frac{3}{2}}))}] \\
    & = \sum_{\bigma \in B(\boldsymbol{\pi}^*,r) } \exp \left( - \frac{\rho^2}{4} \sum_{1 \leq i \neq j \leq p} \left(  \binom{n}{2} - \binom{d_{ij}(\bigma)}{2} \right) + O(n (\log n)^\frac{3}{4}) \right) \\
    & = e^{O(n (\log n)^\frac{3}{4})}\sum_{\bigma \in B(\boldsymbol{\pi}^*,r) } \exp \left( - \frac{\rho^2}{4} \sum_{1 \leq i \neq j \leq p} \left(  \frac{n^2 - d_{ij}(\bigma)^2}{2} \right) \right).
    \end{split}
\end{equation}

By definition of $B(\bi^*, r)$, we have $d_{ij}(\bigma) \geq (1-r)n$ for any $\bigma \in B(\bi^*, r)$ and $1 \leq i \neq j \leq p$. Thus,
\begin{equation} \label{1side1}
\begin{split} 
    \E&[Z_r(\bi^*) \1_{\mathcal{C}}] \\
    &\leq \sum_{(1-r)n \leq (d_{ij}) \leq n} F(d_{ij})\exp \left( - \frac{\rho^2}{4} \sum_{1 \leq i \neq j \leq p} \left(  \frac{n^2 - d_{ij}(\bigma)^2}{2} \right) + O(n (\log n)^\frac{3}{4}) \right) \\
    & \leq \max_{(\alpha_{ij}) \in [(1-r), 1]^{\binom{p}{2}}} F(n \alpha_{ij})\exp \left(- \frac{1}{p}n \log n\sum_{1 \leq i \neq j \leq p} (1 + \eta) (1 - \alpha_{ij}^2) + O(n (\log n)^\frac{3}{4}) \right) \\
    \end{split}
\end{equation}

where $F$ was defined in Theorem \ref{permcount}, and we have substituted $\alpha_{ij} = \frac{d_{ij}}{n}$. Therefore, applying Corollary \ref{usable}, 

\begin{equation}
\begin{split}
    \log  \E &[Z_r(\bi^*) \1_{\mathcal{C}}] \\
    &\leq n \log n   \max_{(\alpha_{ij}) \in [(1-r), 1]^{\binom{p}{2}}} \left( 
    \frac{1}{p}\sum_{1 \leq i \neq j \leq p} (1 - \alpha_{ij} - (1 + \eta) (1 - \alpha_{ij}^2)) \right)+ O(n (\log n)^\frac{3}{4}) \\
     & \leq (p-1) n \log n  \max_{(\alpha_{ij}) \in [(1-r), 1]^{\binom{p}{2}}} \frac{1}{p(p-1)} \sum_{1 \leq i \neq j \leq p} \phe(1 - \alpha_{ij})  + O(n (\log n)^\frac{3}{4}) \\
     & = (p-1) n \log n \max_{r' \in [0, r]} \phe(r') + O(n (\log n)^\frac{3}{4}).
     \end{split}
\end{equation}

proving the first half of the proposition. 

In the same way, 

\begin{equation}
    \E[(Z - Z_r(\bi^*))\1_{\mc{C}}] = e^{O(n (\log n)^\frac{3}{4})}\sum_{\bigma \notin B(\boldsymbol{\pi}^*,r) } \exp \left( - \frac{\rho^2}{4} \sum_{1 \leq i \neq j \leq p} \left(  \frac{n^2 - d_{ij}(\bigma)^2}{2} \right) \right).
\end{equation}

This time, we will use the fact that if $\bigma \notin B(\bi^*, r)$, then for some $j > 1$, $\frac{d_{1j}(\bigma)}{n} \leq 1 - \frac{r}{p-1}$. As a result, setting $\mc{A} = \{(\alpha_{ij})_{1 \leq i < j \leq p} \in [0, 1]^{\binom{p}{2}} | \exists 2 \leq j \leq p, \alpha_{1j} \leq 1-\frac{r}{p-1}  \}$,

\begin{equation}
    \E[(Z - Z_r(\bi^*))\1_{\mc{C}}] \leq \max_{(\alpha_{ij}) \in \mc{A}} F(n \alpha_{ij})\exp \left(- \frac{1}{p}n \log n\sum_{1 \leq i \neq j \leq p} (1 + \eta) (1 - \alpha_{ij}^2) + O(n (\log n)^\frac{3}{4}) \right)
\end{equation}

as in (\ref{1side1}). Again upper bounding $F$ via Corollary \ref{usable}, we thus obtain

\begin{equation}
    \begin{split}
        \log \E[(Z - Z_r(\bi^*))\1_{\mc{C}}] &  \leq (p-1) n \log n  \max_{(\alpha_{ij}) \in \mc{A}} \frac{1}{p(p-1)} \sum_{1 \leq i \neq j \leq p} \phe(1 - \alpha_{ij})  + O(n (\log n)^\frac{3}{4}) \\
        & = (p-1) n \log n \left( \frac{p-1}{p} \max_{r' \in [0,1]} \phe(r') + \frac{1}{p} \max_{r' \in [\frac{r}{p-1}, 1] } \phe(r') \right) + O(n (\log n)^\frac{3}{4})
    \end{split}
\end{equation}
concluding the proof.
 \qed
}
\section{Proof of Lemma \ref{concentration}}
We will rely upon the following concentration theorem:

\begin{theorem}
    Let $(X_1, \ldots, X_N)$ be a standard centred normal vector, and let $f: \R^N \to \R$ be an $L$-Lipschitz function. Then, if $t > 0$,

    \begin{equation}
        \Proba(|f(X) - \E[f(X)]| \geq t) \leq 2e^{-\frac{t^2}{2 L^2}}
    \end{equation}
\end{theorem}
A proof of this theorem can be found in \cite{concentration}.

Note that this remains almost true if $(X_1, \ldots, X_N)$  is a centred normal vector with covariance matrix $\Sigma$, such that $\|\Sigma \|_{op} = 1 + o(1)$. Specifically, in this case, applying the theorem to the standard normal vector $\tilde{X} = \Sigma^{-\frac{1}{2}} X$ and to $g = f(\Sigma^{\frac{1}{2}} \cdot)$, 

\begin{equation}
    \Proba(f(X) - \E[f(X)] \geq t) = \Proba(g(\tilde{X}) - \E[g(\tilde{X})] \geq t) \leq 2e^{-\frac{t^2}{2 L^2}(1 + o(1))}
\end{equation}

as $g$ is a $L(1 + o(1))$-Lipschitz function.

Here, $\log Z$ is a function $F$ of the normal random variables $(G^{(e)}_i)_{e \in E, 1 \leq i \leq p}$; in order to apply this, we therefore want to evaluate $\nabla F$. Thus, if $e \in E$ and $1 \leq i \leq p$:

\begin{equation}
\begin{split}
    \partial_{(e, i)}F((G_e^{(i)})) & = \frac{1}{Z} \partial_{e, i} Z((G_e^{(i)})) \\
    & = \frac{1}{Z}  \sum_{\bigma \in \mathcal{S}_n^{p-1}} 2 \beta \left( \sum_{{j \neq i}} (G_{\sigma_{ij}(e)}^{(j)} - G_e^{(j)})  \right) e^{-\beta V(\bigma)} \\
    & = 2 \beta \sum_{e' \in E, j \neq i}G_{e'}^{(j)}\left( \1_{e = e'} + \1_{e \neq e'} \sum_{\underset{\sigma_{ij}(e) = e'}{\bigma \in \mathcal{S}_n^{p-1}}}  \frac{e^{-\beta V(\bigma)}}{Z} \right) \\
    & \defeq 2 \beta \sum_{e' \in E, j \neq i}G_{e'}^{(j)} \alpha_{(e, i), (e', j)}
    \end{split}
\end{equation}
(We set $\alpha_{(e, i), (e', j)} = 0$ when $i=j$.) Defined as such, we may notice that $0 \leq \alpha_{(e, i), (e, j)} \leq 1 $, and that for any fixed $e', j$ (resp. any fixed $e, i$),

\begin{equation}
    \sum_{e \in E, 1 \leq i \leq p} \alpha_{(e, i), (e', j)} = 2 \text{ resp. }\sum_{e' \in E} \alpha_{(e, i), (e', j)} = 2
\end{equation}

This means that, defining the matrix $A$ by $A_{(e, i), (e', j)} = \alpha_{(e, i), (e', j)}$, the matrix $\dfrac{A}{2 p}$ is doubly stochastic. We may use this in order to bound $\nabla F$:

\begin{equation}
\begin{split}
    \|\nabla F(G) \|^2 & = 4 \beta^2 \sum_{e \in E} \sum_{i=1}^p (\partial_{(e, i)}F(G))^2\\
     & = 4 \beta^2 \sum_{e \in E} \sum_{i=1}^p \left( \sum_{e' \in E, j \neq i}G_{e'}^{(j)} \alpha_{(e, i), (e, j)} \right)^2 \\
    & = 4 \beta^2 \| A G\|_F^2 \leq 16 p^2 \beta^2 \|G\|_F^2
\end{split}
\end{equation}
since $\|A\|_{op} \leq 2p$.

 As a result, for $M > 0$, we will define $p_M$ to be the orthogonal projection onto the $\| \cdot \|_F$ ball of radius $M$: if we truncate our function to obtain $F_M(G) = F(p_M(G))$ for $M > 0$, then $F_M$ is $(4 p \beta M)$-Lipschitz. In this instance, we will take $M = pn^2$.

Consequently, if $K \leq n^\frac{3}{2}$, for large enough $n$,

\begin{equation} \label{concentrationf}
\begin{split}
    \Proba(\log Z - \E[\log Z] > K) & \leq \Proba(\|G \|_{F} \geq M) + \Proba(F_M(G) - \E[F_M(G)] > K) \\
    & \leq \exp(-C n^2) + 2e^{- \frac{K^2}{16 p^2 \beta^2 M^2}(1 + o(1)) } \\
    & \leq 2\exp(- C' \frac{K^2}{n \log n})
\end{split}
\end{equation}
where $C, C'$ are independent from $n$.
\qed
{
\begin{remark} \label{zrpoint}
    We may prove the same theorem for $\log Z_r(\bi^*)$ instead of $\log Z$ in the same way; the changes required are minimal.
\end{remark}
}
\section{Proof of Lemma \ref{technical}} \label{technicalproof}
Set

\begin{equation}
    U = \{ \bigma \in \mathcal{S}_n^{p-1}, \forall 1 \leq i \neq j \leq p, d_{ij}(\bigma) \leq 2 \}
\end{equation}
Then, by Corollary \ref{usable}, $|U| \geq \frac{1}{2}(n!)^{p-1}$.
We will estimate 

\begin{equation}
    N = \# \{ \bigma \in U, \beta V_{\mathrm{off}}(\bigma) \leq -c (p-1) n \log n \}
\end{equation}

where $c = 2\sqrt{1 + \eta} > 0$.

Here, $V_{\mathrm{off}}(\bigma) = G^T M_{\bigma} G$ for some matrix $M_{\bigma}$ defined in (\ref{defmsigreal}) such that, for any $\bigma \in U$,

\begin{equation}
    \Tr((\Sigma M_{\bigma})^2) = p(p-1)\binom{n}{2}(1 + O(\rho)).
\end{equation}

 Thus, we may apply Proposition \ref{mass2}:

\begin{equation}
    \begin{split}
        \E[N] 
        & = |U| \Proba\left(V_{\mathrm{off}}(\bigma)\leq -\frac{c}{\beta} (p-1) n \log n \right) \\
        & \sim |U| \cdot \sqrt{\frac{8p}{c^2 \pi}} \sqrt{\frac{n}{\log n}} \exp \left( - \frac{c^2 (p-1) n \log n}{4 (1 + \eta)}  + O(\sqrt{n \log^3 n}) \right). \\
        & = \exp \left((p-1) n \log \frac{n}{e} - (p-1) n \log n + O(\sqrt{n \log^3 n}) \right) \\
        & = \exp \left( O(n) \right).
    \end{split}
\end{equation}

We would like to show that $\E[N^2] = \exp O(n)$. Assuming this, by the Paley-Zygmund inequality,

\begin{equation}
    \Proba(N > 0) \geq \frac{\E[N]^2}{\E[N^2]} = \exp(- O(n))
\end{equation}

However, if $N> 0$, there exists $\bigma \in U$ such that

\begin{equation}
    \begin{split}
        \log Z & \geq - \beta V(\bigma)  \\
        & \geq - \rho \beta p(p-1) \binom{n}{2} - \beta V_{\mathrm{off}}(\bigma) + \beta(V_{\mathrm{diag}}(\bigma) - \E[V_{\mathrm{diag}}(\bigma)]) \\
       & \geq - 2 (1 + \eta) (p-1) n \log n +   2 \sqrt{1 + \eta} (p-1) n \log n + \beta(V_{\mathrm{diag}}(\bigma) - \E[V_{\mathrm{diag}}(\bigma)]) \\
       & \geq - (p-1)\eta n \log n - (p-1) \eta^2 n \log n + \beta(V_{\mathrm{diag}}(\bigma) - \E[V_{\mathrm{diag}}(\bigma)]) \\
       & = (p-1)|\eta| n \log n - (p-1) n (\log n)^{\frac{3}{5}} + \beta(V_{\mathrm{diag}}(\bigma) - \E[V_{\mathrm{diag}}(\bigma)])
    \end{split}
\end{equation}

Furthermore, by Proposition \ref{weaklemma},
\begin{equation}
    \Proba(\exists \bigma \in \mathcal{S}_n^{p-1}, |V_{\mathrm{diag}}(\bigma) - \E[V_{\mathrm{diag}}(\bigma)]| \geq (\log n)^{\frac{1}{4}}n^{\frac{3}{2}}) = \Proba(\overline{\mathcal{C}}) \leq \exp (- c n \sqrt{\log n}(1 + o(1))).
\end{equation}
where $\overline{\mathcal{C}}$ is the complement of the event $\mathcal{C}$. This will mean that $\Proba((\log Z) \1_{\mathcal{C}} \geq - (p-1)\eta n \log n - (p-1) n (\log n)^{\frac{3}{5}} \geq \exp(- O(n)) - \Proba(\overline{\mathcal{C}}) = \exp(- O(n)) $, concluding the proof of Lemma \ref{technical}.

We now just need to show that $\E[N^2] = \exp \left(O(n) \right)$. Indeed, by Proposition \ref{mass2},

\begin{equation}
    \begin{split}
        \E[N^2] & = \sum_{\bigma, \bigma' \in U} \Proba\left(V_{\mathrm{off}}(\bigma)\leq -\frac{c}{\beta} (p-1) n \log n ,V_{\mathrm{off}}(\bigma') \leq - \frac{c}{\beta} (p-1) n \log n\right) \\
        & \leq  \sum_{\bigma, \bigma' \in U}\Proba\left(V_{\mathrm{off}}(\bigma) +V_{\mathrm{off}}(\bigma') \leq -2\frac{c}{\beta} (p-1) n \log n \right) \\
        & \leq \frac{\beta\sqrt{\Tr((\Sigma(M_{\bigma} + M_{\bigma'}))^2)}}{2\sqrt{\pi} c (p-1) n \log n}\exp \left(- \frac{c^2 n^2 \log^2 n}{\beta^2 \Tr((\Sigma(M_{\bigma} + M_{\bigma'}))^2)} + O(\sqrt{n \log^3 n}) \right)(1 + o(1)) \\
        & = \exp \left(- \frac{c^2 (p-1)^2 n^2 \log^2 n}{\beta^2 \Tr((\Sigma(M_{\bigma} + M_{\bigma'}))^2)} + O(n) \right)
    \end{split}
\end{equation}

To continue, we will note that, as we showed in (\ref{C17}), we may write

\begin{equation}
\begin{split}
    \Tr((\Sigma(M_{\bigma} + M_{\bigma'}))^2) &\leq 2p(p-1)\binom{n}{2} + 2\sum_{1 \leq i \neq j \leq p} \binom{c_{ij}(\bigma, \bigma')}{2} + O(\rho n^2) \\
    & \leq \left( p(p-1) n^2 + \sum_{1 \leq i \neq j \leq p} c_{ij}(\bigma, \bigma')^2 \right) (1 + O(\rho))
    \end{split}
\end{equation}

where $c_{ij}(\bigma, \bigma') = \# \{u \in \llb 1, n \rrb, \sigma_{ij}(u)=\sigma'_{ij}(u) \} = n \cdot ov(\pi_i^*\sigma_i^{-1} \sigma_i'(\pi_i^*)^{-1},\pi_j^*\sigma_j^{-1}\sigma'_j (\pi_j^*)^{-1})$. Note that this definition parallels that of $d_{ij}(\bigma) = n \cdot ov(\pi_i^* \sigma_i^{-1}, \pi_j^* \sigma_j^{-1})$; in particular, we will point out that, for fixed $\bigma \in \mathcal{S}_n^{p-1}$, if $0 \leq (c_{ij})_{i \neq j} \leq n$,

\begin{equation}
    \# \{ \bigma' \in \mathcal{S}_n^{p-1}, \forall 1 \leq i \neq j \leq p, c_{ij}(\bigma, \bigma') = c_{ij}\} = F(c_{ij})
\end{equation}

where $F$ was defined in Theorem \ref{permcount}. Thus:

\begin{equation}
    \begin{split}
        \E[N^2] & \leq \sum_{\bigma \in U} \sum_{\bigma' \in U} \exp \left(- \frac{2 (p-1) n \log n}{1 + \frac{1}{p(p-1) n^2}\sum_{1 \leq i \neq j \leq p} c_{ij}(\bigma, \bigma')^2} + O(n)\right) \\
        & \leq \sum_{\bigma \in \mathcal{S}_n^{p-1}} \sum_{\bigma' \in \mathcal{S}_n^{p-1}} \exp \left(- \frac{2 (p-1) n \log n}{1 + \frac{1}{p(p-1) n^2}\sum_{1 \leq i \neq j \leq p} c_{ij}(\bigma, \bigma')^2} + O(n)\right) \\
        & \leq (n!)^{p-1}\sum_{\bigma' \in \mathcal{S}_n^{p-1}}\exp \left(- \frac{2(p-1) n \log n}{1 + \frac{1}{p(p-1) n^2}\sum_{1 \leq i \neq j \leq p} d_{ij}(\bigma')^2} + O(n)\right) 
    \end{split}
\end{equation}
and so

\begin{equation}
    \begin{split}
        \E[N^2] & \leq (n!)^{p-1} \sum_{0 \leq (d_{ij})_{1 \leq i \neq j \leq p} \leq n} F(d_{ij}) \exp \left(- \frac{2 (p-1) n \log n}{1 + \frac{1}{p(p-1) n^2}\sum_{1 \leq i \neq j \leq p} d_{ij}^2} + O(n) \right) \\
    \end{split}
\end{equation}

However, instead of using Corollary \ref{usable} to bound $F(d_{ij})$, we will need to use the full Theorem \ref{permcount}. For any $\tau \in \mathcal{S}_p$,

\begin{equation}
\begin{split}
    \E[N^2 ] & \leq \sum_{0 \leq (d_{ij})_{1 \leq i \neq j \leq p} \leq n} \frac{(n!)^{p-1}}{(\ds\max_{j < 2} d_{\tau(j)\tau(2)})!\ldots (\ds\max_{j < p} d_{\tau(j) \tau(p)})!} \exp \left(- \frac{2 (p-1) n \log n}{1 + \frac{1}{p(p-1)n^2}\sum_{1 \leq i \neq j \leq p} d_{ij}^2} + O(n) \right) \\
    & \leq (n!)^{2(p-1)}\sum_{0 \leq (d_{ij})_{1 \leq i \neq j \leq p} \leq n}  \exp \left(-  \sum_{i=2}^p D_i(\tau) \log_+(D_i(\tau)) - \frac{2 (p-1) n \log n}{1 + \frac{1}{p(p-1)n^2}\sum_{1 \leq i \neq j \leq p} d_{ij}^2} + O(n) \right) 
    \end{split}
\end{equation}

where we have set $D_i(\tau) = \ds\max_{j < i} d_{\tau(j) \tau(i)}$ for $2 \leq i \leq p$.

Note that, in fact, $D_i(\tau) \log_+(D_i(\tau)) = D_i(\tau) \log n + O(n)$. Thus, in order to show that $\E[N^2] \leq e^{O(n)}$, we will prove that, for any values of $(d_{ij})$, there exists $\tau$ such that

\begin{equation}
     \sum_{i=2}^p D_i(\tau) \log n + \frac{2 (p-1) n \log n}{1 + \frac{1}{p(p-1)n^2}\sum_{1 \leq i \neq j \leq p} d_{ij}^2} \geq 2(p-1) n \log n.
\end{equation}
Setting $\alpha_{ij} = \frac{d_{ij}}{n}$ and $A_i(\tau) = \frac{1}{n}D_i(\tau)$, this is equivalent to showing that there exists $\tau$ such that

\begin{equation}
    \frac{1}{p-1} \sum_{i=2}^p A_i(\tau) \geq \frac{\frac{2}{p(p-1)}\sum_{1 \leq i \neq j \leq p} \alpha_{ij}^2}{1 + \frac{1}{p(p-1)}\sum_{1 \leq i \neq j \leq p} \alpha_{ij}^2}.
\end{equation}

This fact is true, but its proof is quite technical: we defer it to Appendix \ref{spanningtree}. Assuming this, we indeed have $\E[N^2] \leq e^{O(n)}$ and so Lemma \ref{technical} is proven.
\qed
\section{Proof of Lemma \ref{exact}}

The main difference between this argument and the proof of Proposition \ref{energy} is that, in this instance, we cannot ignore the fluctuations of $V_{\mathrm{diag}}$: we will be directly evaluating $\sum_{\bigma \in B(\bi^*, \eps)}\E[e^{- \beta V(\bigma)}]$. To this end, note that, for any $\bigma \in \mathcal{S}_n^{p-1}$, we may write $V(\bigma) = -G^T M'_{\bigma} G$, as a quadratic form over $\R^E \otimes \R^{p}$, where

\begin{equation} \label{defmsig}
    (M'_{\bigma})_{(e, i), (e', j)} =  \1_{i \neq j}(\1_{e'=\sigma_{ij}(e)} - \1_{e'=e})
\end{equation}

for any $e, e' \in E, 1 \leq i , j \leq p$. As we saw in Appendix \ref{quad2}, the covariance matrix $\Sigma$ of $G$ is equal to $(1 - \rho)\Id + \rho \tilde{J}$, where $\|\tilde{J}\|_{op} \leq p-1$.{ Therefore, for any $\bigma \in \mathcal{S}_n^{p-1}$, as in  (\ref{crude}),

\begin{equation}
    \begin{split}
        \Tr((\Sigma M_{\bigma}')^k) & = \Tr((\Sigma^{\frac{1}{2}} M_{\bigma}' \Sigma^{\frac{1}{2}})^k)\\
        &\leq \Tr(( \Sigma M_{\bigma}')^2)(2p)^{k-2}.\\
    \end{split}
\end{equation}

Furthermore, again by (\ref{traces}):
\begin{equation}
    \begin{split}
        \Tr(( \Sigma M_{\bigma}')^2) & = \Tr((M_{\bigma}')^2)(1 + o(1)) + 2 \rho (1-\rho) \Tr(\tilde{J} M_{\bigma}^2) + \rho^2\Tr((\tilde{J} M_{\bigma}')^2).
    \end{split}
\end{equation}
}

Here, by Lemma \ref{Dij}:

\begin{equation}
    \begin{split}
        \Tr((M_{\bigma}')^2) &= \sum_{(e, i), (e', j)} \1_{i \neq j}(\1_{e'=\sigma_{ij}(e)} - \1_{e'=e})^2 \\
        & = 2 \sum_{1 \leq i \neq j \leq p} \left( \binom{n}{2} - D_{ij}(\bigma)\right) \\
        & = 2 \sum_{1 \leq i \neq j \leq p} \left( \binom{n}{2} - \binom{d_{ij}(\bigma)}{2} \right)(1 + o(1)).
    \end{split}
\end{equation}
Similarly:

\begin{equation}
    \begin{split}
        \Tr((\tilde{J}M_{\bigma}')^2) & = \sum_{(e, i), (e', j), (e'', k)} \left(\sum_{l=1}^{p} \1_{l \neq k}(1_{e'' = \sigma_{lk}(e)} - 1_{e'' = e})\right)\left(\sum_{l=1}^{p} \1_{l \neq j}(1_{e' = \sigma_{lj}(e'')} - 1_{e' = e''})\right) \\
        & \leq 2 \sum_{(e, i), (e', j)} \sum_{1 \leq k, l, l' \leq p } \1_{\sigma_{lk}(e) \neq e} \1_{\sigma_{jl'}(e') \neq e'} \1_{e' \in \{ \sigma_{l'j}\sigma_{lk}(e), \sigma_{l'k}(e), \sigma_{lj}(e), e  \}} \\
        & \leq  2 \sum_{(e, i)} \sum_{1 \leq k, l, l' \leq p } \1_{\sigma_{lk}(e) \neq e}  \left( \sum_{ (e', j)}  \1_{e' \in \{ \sigma_{l'j}\sigma_{lk}(e), \sigma_{l'k}(e), \sigma_{lj}(e), e  \}} \right) \\
        & \leq 8 p^2 \sum_{e \in E}\sum_{1 \leq k, l \leq p } \1_{\sigma_{lk}(e) \neq e} \\\
        & = 4 p^2 (1 + o(1)) \Tr((M_{\bigma}')^2)
    \end{split}
\end{equation}

and so

\begin{equation}
    \Tr(( \Sigma M_{\bigma}')^2) =  2 \sum_{1 \leq i \neq j \leq p} \left( \binom{n}{2} - \binom{d_{ij}(\bigma)}{2} \right)(1 + o(1)).
\end{equation}

{
Finally, by Cauchy-Schwarz,

\begin{equation}
    |\Tr(\tilde{J} (M_{\bigma}')^2)| \leq \sqrt{\Tr((M_{\bigma}')^2)\Tr((\tilde{J} M_{\bigma}')^2)} = O(1) \Tr((M_{\bigma}')^2).
\end{equation}
}
As a result, if $\eps > 0$, by Lemma \ref{weakmass},

\begin{equation}
    \begin{split}
        \Proba& (\exists \bigma \in B(\bi^*, \eps) \setminus \{ \bi^*  \}, V(\bigma) \leq 0) \\
        &\leq \sum_{\bigma \in B(\bi^*, \eps)}\Proba(V(\bigma) \leq 0) \\
        & \leq \sum_{\bigma \in B(\bi^*, \eps)\setminus \{ \bi^*  \}} \exp \left( - \frac{\E[V(\bigma)]^2}{4 \Tr((\Sigma M_{\bigma}')^2)} \right) \\
        & = \sum_{\bigma \in B(\bi^*, \eps)\setminus \{ \bi^*  \}}\exp \left( - \frac{\rho^2}{8} \sum_{1 \leq i \neq j \leq p} \left( \binom{n}{2} - \binom{d_{ij}(\bigma)}{2} \right)(1 + o(1)) \right) \\
        & = \sum_{ \underset{(d_{ij}) \neq (n)}{(1 - p(p-1) \eps)n < (d_{ij})_{1 \leq i \neq j \leq p} \leq n}} F((d_{ij}))\exp \left( - \frac{\rho^2}{8} \sum_{1 \leq i \neq j \leq p} \left( \binom{n}{2} - \binom{d_{ij}}{2} \right)(1 + o(1)) \right).
    \end{split}
\end{equation}

since $B(\bi^*, \eps) \subseteq \{ \bigma \in \mathcal{S}_n^{p-1}, \forall 1 \leq i \neq j \leq p, d_{ij}(\bigma) > 1 - p(p-1) \eps \}$. By Corollary $\ref{usable}$, thus,

\begin{equation}
    \begin{split}
        \Proba& (\exists \bigma \in B(\bi^*, \eps) \setminus \{ \bi^*  \}, V(\bigma) \leq 0)\\
        & \leq \sum_{ \underset{(d_{ij}) \neq (n)}{(1 - p(p-1) \eps )n < (d_{ij}) \leq n}} \exp \left( \frac{1}{p} \sum_{1 \leq i \neq j \leq p} (n - d_{ij}) \log n - \frac{\rho^2}{8} \sum_{1 \leq i \neq j \leq p} \left( \binom{n}{2} - \binom{d_{ij}}{2} \right)(1 + o(1)) \right).
    \end{split}
\end{equation}

since $\frac{n!}{(d_{ij})!} \leq n^{n-d_{ij}} = e^{(n - d_{ij}) \log n}$. Setting $\alpha_{ij} = \dfrac{d_{ij}}{n}$, we may simplify this:

\begin{equation}
    \begin{split}
        \Proba& (\exists \bigma \in B(\bi^*, \eps) \setminus \{ \bi^*  \}, V(\bigma) \leq 0)\\
        & \sum_{ \underset{(d_{ij}) \neq (n)}{(1 - p(p-1) \eps )n < (d_{ij}) \leq n}}\leq \exp \left(\frac{1}{p} n \log n \sum_{1 \leq i \neq j \leq p} \left((1 - \alpha_{ij}) - (1 + \eta)(1 + o(1))\frac{1 - \alpha_{ij}^2}{2} \right)\right) \\
        & \leq \sum_{ \underset{(d_{ij}) \neq (n)}{(1 - p(p-1) \eps )n < (d_{ij}) \leq n}}\ \exp \left(\frac{1}{p} n \log n \sum_{1 \leq i \neq j \leq p} (1 - \alpha_{ij})\left(1 - (1 + \eta)(1 + o(1))\frac{2 - p(p-1)\eps}{2}\right) \right) \\
        & = \sum_{ \underset{(d_{ij}) \neq (n)}{(1 - p(p-1) \eps )n < (d_{ij}) \leq n}} \exp \left(- \frac{1}{p} n \log n c_{\eps} \sum_{1 \leq i \neq j \leq p} (1 - \alpha_{ij}) \right) = o(1)
    \end{split}
\end{equation}

where $c_{\eps} > 0$ as soon as $\eps < \frac{\eta}{p(p-1)}$.
\qed

{
\section{Proof of Theorem \ref{mainthm} continued: extra technical step}\label{sec}

Recall that we need to prove that

\begin{equation} \label{extranec}
    \max_{\bi \in \mathcal{S}_n^{p-1}} Z_r(\bi) \ll_{\Proba} Z
\end{equation}
where $r < 1$ has been fixed.

Define $\bi^e = \underset{\bi \in \mc{S}_n^{p-1}}{\arg \max} Z_r(\bi)$. We also set $\mc{Y} = \sigma((\bi^*)^{-1}(G))$ the $\sigma$-algebra of all observable random variables.

Let $\eps > 0$, and $\mc{X}_\eps$ be the event $\{Z_r(\bi^e) \geq \eps Z\}$. In this sense, showing (\ref{extranec}) is equivalent to showing that $\Proba(\mc{X}_\eps) = o(1)$ for any such $\eps$.

We begin with the following lemma.

\begin{lemma}
\begin{equation}
    \Proba(\bi^e \in B(\bi^*, r)|\mc{X}_\eps) \geq \eps.
\end{equation}
\end{lemma}
\begin{proof}
    We compute:

    \begin{equation}
        \begin{split}
            \Proba(\bi^e \in B(\bi^*, r) \wedge \mc{X}_\eps) & = \E[\1_{\bi^e \in B(\bi^*, r)} \1_{Z_r(\bi^e) \geq \eps Z}] \\
            & = \E[\1_{\bi^* \in B(\bi^e, r)} \1_{\Proba_{post}(B(\bi^e, r)) \geq \eps}] \\\
            & = \E\left[\E\left[\1_{\bi^* \in B(\bi^e, r)} |  \mc{Y} \right] \1_{\Proba_{post}(B(\bi^e, r)) \geq \eps} \right] \\
            & = \E\left[\Proba_{post}(B(\bi^e, r)) \1_{\Proba_{post}(B(\bi^e, r)) \geq \eps} \right] \\
            & \geq \eps \Proba(\mc{X}_\eps).
        \end{split}
    \end{equation} \qed
\end{proof}

This means that, defining the new event $\mc{X}_\eps^{int} = \{Z_r(\bi^e) \geq \eps Z\} \wedge \{ \bi^e \in B(\bi^*, r)\})$, it is sufficient for us to show that $\Proba(\mc{X}_\eps^{int}) = o(1)$.
To achieve this, we will show the following lemma.

\begin{lemma} \label{awful}
    There exists an estimator $\hat{Y} \in \mc{P}(\mc{S}_n^{p-1})$ such that

\begin{equation} \label{toshow1}
    \Proba(\bi^* \in \hat{Y} \subseteq B(\bi^*, r) | \mc{X}_\eps^{int}) \geq \exp(-O(n)).
\end{equation}
\end{lemma}

Assume this lemma. Then,

\begin{equation}
    \Proba(\bi^* \in \hat{Y} \subseteq B(\bi^*, r)) \geq \Proba(\mc{X}_\eps^{int}) \exp(-O(n)).
\end{equation}

However, we know that

\begin{equation}
    \begin{split}
        \Proba(\bi^* \in \hat{Y} \subseteq B(\bi^*, r)) & = \E[\1_{\bi^* \in \hat{Y}}\1_{\hat{Y} \subseteq B(\bi^*, r)}] \\
        & \leq \E[\1_{\bi^* \in \hat{Y}} \1_{\Proba_{post}(\hat{Y}) \leq \frac{Z_r(\bi^*)}{Z}}]\\
        & \leq \Proba(\frac{Z_r(\bi^*)}{Z} \geq \delta_n) + \E[\1_{\bi^* \in \hat{Y}} \1_{\Proba_{post}(\hat{Y}) \leq \delta_n}]
    \end{split}
\end{equation}

where $\delta_n = \exp(-n (\log n)^{\frac{3}{5}})$.

By applying Lemma \ref{concentration} (as well as Remark \ref{zrpoint}), we may see that for some $c > 0$,

\begin{equation}
    \Proba(\frac{Z_r(\bi^*)}{Z} \geq \delta_n) \leq 2\exp(-c n (\log n)^{\frac{1}{5}})
\end{equation}

since $\E[\log Z] \sim (p-1) n (\log n)^{\frac{4}{5}}$ and, by the same argument as in (\ref{idkgarbage}),  

\begin{equation}
    \E[\log Z_r(\bi^*)] \sim \E[\log Z_r(\bi^*)\1_{\mc{C}}] \leq c_r (p-1) n (\log n)^{\frac{4}{5}}(1 + o(1))
\end{equation}
for some $0 \leq c_r < 1$.

Furthermore,

\begin{equation}
    \begin{split}
        \E[\1_{\bi^* \in \hat{Y}} \1_{\Proba_{post}(\hat{Y}) \leq \delta_n}] & = \E \left[\E \left[\1_{\bi^* \in \hat{Y}} |\mc{Y}\right] \1_{\Proba_{post}(\hat{Y}) \leq \delta_n}\right] \\
        & = \E[\Proba_{post}(\hat{Y})\1_{\Proba_{post}(\hat{Y}) \leq \delta_n}] \\
        & \leq \delta_n.
    \end{split}
\end{equation}
 This all implies that 

\begin{equation}
    \Proba(\mc{X}_\eps^{int}) \leq \exp O(n) (\exp(-c n (\log n)^{\frac{1}{5}}) + \delta_n) = o(1)
\end{equation}

which is what we wanted. 

We now just need to prove Lemma \ref{awful}.

\begin{proof}[Proof of Lemma \ref{awful}]
    Pick $I_1, \ldots,  I_p \subseteq \llb 1, n \rrb$ to be uniformly random subsets of $\llb 1, n \rrb$, independently from each other and everything else; we then define

    \begin{equation}
        \hat{Y} = \{ \bi \in \mc{S}_{n}^{p-1}, \forall 1 \leq l \leq p, \forall i \in I_l, \bi_l(i) = \bi_l^e(i) \}.
    \end{equation}
    Assume that $\mc{X}_\eps^{int}$ holds. Then, if $I_l = \{ i \in \llb 1, n \rrb, \bi^e_l(i) = \bi_l^*(i) \}$ for any $1 \leq l \leq p$ (which occurs with probability at least $\exp (- O(n))$), we will in fact have

    \begin{equation}
        \bi^* \in \hat{Y} \subseteq B(\bi^*, r).
    \end{equation}

    Indeed, given that $\mc{X}_\eps^{int}$ is verified, we may set $J = \bigcap_{1 \leq l \leq p} I_l$, and $|J| \geq (1-r)n$; thus, if $\bi \in \hat{Y}$,

    \begin{equation}
        ov(\bi, \bi^*) \geq \frac{|J|}{n} = 1-r
    \end{equation}

    and $\hat{Y} \subseteq B(\bi^*, r)$.
\end{proof}}
\section{Proof of Lemma \ref{auttrees}}

We begin by assuming that $H$ is connected. We will prove by recursion that, for any $1 \leq i \leq |V|$, there exists  $X \subseteq V$ with $|X| \geq  i$ such that the cardinal of  $\{ \sigma|_{X}, \sigma \in \Aut H \}$ is at most equal to

\begin{equation}
     |V|\prod_{\underset{N_u \subseteq X}{u \in X}} (d_u)!
\end{equation}

where, for any vertex $u$, $N_u$ is the set of neighbours of $u$.

This is clear for $i=1$. Furthermore, if we have such a set $X \neq V$, then there exist $u \in X$ such that $N_u$ is not contained within $X$. However, if $\sigma \in \Aut H$, $\sigma(N_{u}) = N_{\sigma(u)}$ and so, given the value of $\sigma(u)$, there are at most $d_u!$ possible values for $\sigma|_{N_u}$. Thus, the property remains true for $X' = X \cup N_u$. By recursion, this concludes the case where $H$ is connected.

Now, we move on to the general case: assume that $H$ splits into connected components $(H_i)_{1 \leq i \leq l}$ with vertex sets $(V_i)_{1 \leq i \leq l}$. From the connected case, we deduce that the number of automorphisms of $H$ which send each $H_i$ to itself, is at most equal to 

\begin{equation}
    \left(\prod_{i=1}^l |V_i| \right)\prod_{v \in V}(d_v!).
\end{equation}

However, since $\sum_{i=1}^l |V_i| = |V|$, by the AM-GM inequality,

\begin{equation}
    \prod_{i=1}^l |V_i| \leq \left(\frac{1}{l} \sum_{i=1}^l |V_i|\right)^l = \left(\left(\frac{|V|}{l} \right)^{\frac{l}{|V|}} \right)^{|V|} \leq \left(e^{\frac{1}{e}} \right)^{|V|},
\end{equation}
concluding the proof.
\qed
\section{An extra theorem} \label{spanningtree}

At the end of Appendix \ref{technicalproof}, we make use of an inequality which we shall now prove.

\begin{theorem}\label{spanthm}
    Let $0 \leq (x_{ij})_{1 \leq i \neq j \leq p} \leq 1$ be real numbers with $x_{ij} = x_{ji}$ for all $i \neq j$. If $\tau \in \mathcal{S}_p$, and $2 \leq i \leq p$, define $X_{i}(\tau) = \max_{j<i} x_{\tau(j) \tau(i)}$. We also set

    \begin{equation}
        S_*(\tau) = S_*(\tau, (x_{ij})) \defeq \frac{1}{p-1}\sum_{i=2}^p X_{i}(\tau)
    \end{equation}
    
    Then, there exists $\tau \in \mathcal{S}_p$ such that

    \begin{equation}
        S_*(\tau) \geq \frac{\dfrac{2}{p(p-1)}\dsum_{1 \leq i \neq j \leq p} x_{ij}^2}{1 + \dfrac{1}{p(p-1)}\dsum_{1 \leq i \neq j \leq p} x_{ij}^2}
    \end{equation}
\end{theorem}

\begin{remark} \label{spanrk}
    Consider the complete graph $K_p$ over $\llb 1, p \rrb$; we put the weight $x_{ij}$ on the edge linking $i$ to $j$, obtaining a graph $\mathcal{G}$. We construct a  spanning tree of $\mathcal{G}$, by:
    \begin{itemize}
        \item picking $\tau(1)$;
        \item picking $\tau(2)$ and connecting it to $\tau(1)$;
        \item picking $\tau(3)$ and connecting it to $\tau(1)$ or $\tau(2)$ (depending on whether $x_{\tau(1) \tau(3)}$ or $x_{\tau(2) \tau(3)}$) is larger;
        \item and continuing the process in this fashion.
    \end{itemize}
    $X_i(\tau)$ is the $(i-1)$th edge constructed by this process, and so $(p-1)S_*(\tau)$ is the total weight of the tree. Furthermore, if $\tau$ is chosen correctly, this is simply the application of Prim's algorithm; thus, $\max_{\tau \in \mathcal{S}_p} (p-1)S_*(\tau)$ is the total weight of the maximal spanning tree of $\mathcal{G}$. We are therefore claiming that there exists a spanning tree of $\mathcal{G}$, with weight at least

    \begin{equation}
        (p-1) \cdot \frac{ \dfrac{2}{p(p-1)}\dsum_{1 \leq i \neq j \leq p}x_{ij}^2}{1 + \dfrac{1}{p(p-1)}\dsum_{1 \leq i \neq j \leq p}x_{ij}^2}.
    \end{equation}
\end{remark}

\begin{proof}
    We will proceed by recursion over $p \geq 1$. If $p=1$, there is nothing to prove. From now on, we assume that $p \geq 2$.

    Let $x$ be the minimum weight of an edge contained in the maximal spanning tree.
    Then, the graph over $\llb 1, p \rrb$ with edges $\{ \{i\neq j \}, x_{ij} > x \}$ cannot be connected; otherwise, Prim's algorithm would give us a maximal spanning tree with edges whose weights are strictly larger than $x$. Thus, we may partition $\llb 1, p \rrb$ into two (nonempty) subsets $B_1, B_2 \subseteq \llb 1, p \rrb$ such that $x_{ij} \leq x$ for $i \in B_1, j \in B_2$, and there exists a maximal spanning tree which is the union of 
    \begin{itemize}
        \item a maximal spanning tree of the induced subgraph of $\mathcal{G}$ over $B_1$;
        \item a maximal spanning tree of the induced subgraph of $\mathcal{G}$ over $B_2$;
        \item an edge with weight $x$.
    \end{itemize}
Also note that we may assume, without loss of generality, that $x_{ij} \geq x$ for any $i \neq j \in B_1$ or $i \neq j \in B_2$. Indeed, even after increasing the $x_{ij}$ to satisfy this condition, applying Prim's algorithm (with suitable tiebreakers) will still give us the same spanning tree, (since all of the edges of the spanning tree are $\geq x$) but the quantity $\frac{2 \sum_{1 \leq i \neq j \leq p}x_{ij}^2}{1 + \sum_{1 \leq i \neq j \leq p}x_{ij}^2}$ will have increased.
    
    Noting $|B_1| = k_1$ and $|B_2| = k_2 = p-k_1$, this means that

    \begin{equation}
        \begin{split}
            (p-1)S_* & = (k_1-1)S_*((x_{ij})_{i,j \in B_1}) + (k_2 - 1)S_*((x_{ij})_{i,j \in B_2}) + x. \\
        \end{split}
    \end{equation}
However, by induction, we have

    \begin{equation}
         (k_1-1)S_*((x_{ij})_{i,j \in B_1}) + (k_2 - 1)S_*((x_{ij})_{i,j \in B_2})\geq (k_1 - 1) \frac{2 S_1}{1 + S_1} + (k_2 - 1) \frac{2 S_2}{1 +S_2}
    \end{equation}
    where $S_l = \frac{1}{k_l(k_l - 1)} \dsum_{i \neq   j \in B_l} x_{ij}^2$ for $l = 1$ or $2$. 
    (Note that, in our case, $S_l \geq x^2$.) In order to conclude, we thus simply need to show that

    \begin{equation} \label{K7}
        \begin{split}
            (k_1 - 1) \frac{2 S_1}{1 + S_1} + (k_2 - 1) \frac{2 S_2}{1 +S_2} + x & \geq  (p-1)\frac{2 \frac{1}{p(p-1)} \sum_{1 \leq i \neq j \leq p} x_{ij}^2}{1 + \frac{1}{p(p-1)} \sum_{1 \leq i \neq j \leq p} x_{ij}^2} \\
        \end{split}
    \end{equation}

    To simplify notations, we define
    \begin{equation}
        \begin{split}
            \alpha_1 = \frac{k_1 - 1}{p-1}, \alpha_2 = \frac{k_2 - 1}{p-1}, \alpha_3 = \frac{1}{p-1}
        \end{split}
    \end{equation}
    \begin{equation}
        \beta_1 = \frac{k_1(k_1 - 1)}{p(p-1)}, \beta_2 = \frac{k_2(k_2 - 1)}{p(p-1)}, \beta_3 = \frac{2k_1k_2}{p(p-1)}
    \end{equation}
    such that $\alpha_1 + \alpha_2 + \alpha_3 = \beta_1 + \beta_2 + \beta_3 = 1$. By definition of our blocks, we have:

    \begin{equation}
        \frac{1}{p(p-1)} \sum_{1 \leq i \neq j \leq p} x_{ij}^2 \leq \beta_1 S_1 + \beta_2 S_2 + \beta_3 x^2
    \end{equation}

    and so, in order to show (\ref{K7}), we just need to show that

    \begin{equation}
        \alpha_1 \frac{2 S_1}{1 + S_1} + \alpha_2 \frac{2 S_2}{1 + S_2} + \alpha_3 x \geq \frac{2(\beta_1 S_1 + \beta_2 S_2 + \beta_3 x^2)}{1 + \beta_1 S_1 + \beta_2 S_2 + \beta_3 x^2}
    \end{equation}
    
    Expanding out, the above inequality reduces to:

    \begin{equation}
    \begin{split}
        f(S_1, S_2) &\defeq (2\alpha_1 S_1(1 + S_2) + 2\alpha_2 S_2(1 + S_1) + \alpha_3 x (1 + S_1)(1 + S_2))(1 + \beta_1 S_1 + \beta_2 S_2 + \beta_3 x^2) \\
        & - (2\beta_1 S_1 + 2\beta_2 S_2 + 2 \beta_3 x^2)(1 + S_1)(1 + S_2) \\
        & \geq 0.
        \end{split}
    \end{equation}

    We will show that:

    \begin{itemize}
        \item for $x \leq S_1, S_2 \leq 1$, $\dfrac{\partial^2 f}{\partial S_i^2}(S_1, S_2) \leq 0$, i.e. $f$ is concave in both arguments;
        \item $f(x^2, 1) \geq 0$, $f(1, x^2) \geq 0$, $f(1, 1) \geq 0$, and $f(x^2, x^2) \geq 0$,
    \end{itemize}
    proving that indeed $f(S_1, S_2) \geq 0$ for $x^2 \leq S_1, S_2 \leq 1$.

    First of all,

    \begin{equation}
        \begin{split}
            \dfrac{\partial^2 f}{\partial S_1^2}(S_1, S_2) = 2 \alpha_1 \beta_1 (1 + S_2) - 2 \beta_1 (1 + S_2) \leq 0.
        \end{split}
    \end{equation}
    and, by symmetry, $ \dfrac{\partial^2 f}{\partial S_2^2} \leq 0$.
    Secondly, expanding out $f$,
        \begin{equation}
            \begin{split}
                f(x^2, 1) & = 2 \alpha_{3} x(1 - x)^2 + 2 \alpha_3 (1 - \beta_2) x^3 (1 - x)^2 \\
                & + 2 \alpha_2(1 + \beta_2) - 4 \beta_2 + 2 \alpha_3 \beta_2 x + 4 \alpha_1 \beta_2 x^2 + 2 \alpha_3 \beta_2 x^3 - 2 \alpha_2(1 - \beta_2)x^4 \\
                & \geq 2 \alpha_2(1 + \beta_2) - 4 \beta_2 + 2 \alpha_3 \beta_2 + 4 \alpha_1 \beta_2  + 2 \alpha_3 \beta_2  - 2 \alpha_2(1 - \beta_2) = 0
            \end{split}
        \end{equation}
        since $2 \alpha_2(1 + \beta_2) - 4 \beta_2 = \frac{2k_1(k_2 - 1)}{p(p-1)} \geq 0$. 
        
        Thirdly, by symmetry, $f(1, x^2) \geq 0$.
        
        Fourthly,
        \begin{equation}
            \begin{split}
                f(1, 1) & = 4(1 - \alpha_3(1 - x))(2 - \beta_3(1-x^2)) - 8(1 - \beta_3(1 - x^2)) \\
                & = 8 \frac{k_1 k_2}{p(p-1)}(1 - x^2) - \frac{8}{p-1}(1-x) + \frac{8 k_1 k_2}{p(p-1)^2}(1-x)(1-x^2) \\
                & = \frac{8}{p(p-1)^2}(k_1(p-k_1)(p-1)(1-x^2) - p(p-1)(1-x) + k_1(p - k_1)(1-x)(1-x^2))
            \end{split}
        \end{equation}
        If $k_1 \geq 2$ and $k_2 = p-k_1 \geq 2$, then $k_1(p-k_1) \geq p$ and $f(1, 1) \geq 0$. Otherwise $\{k_1, k_2 \} = \{1, p-1 \}$, and 

        \begin{equation}
            \begin{split}
                f(1, 1) & = \frac{8}{p(p-1)^2}((p-1)^2(1-x^2) - p(p-1)(1-x)+(p-1)(1-x)(1-x^2)) \\
                & = \frac{8}{p(p-1)^2}(p-1)(p-2)x^2(1-x) \geq 0.
            \end{split}
        \end{equation}
        Finally, 

        \begin{equation}
            \begin{split}
                f(x^2, x^2) & = (1 + x^2)^2(2(\alpha_1 + \alpha_2)x^2 + \alpha_3 x(1 + x^2) - 2 x^2) \\
                & = (1 + x^2)^2 \alpha_3 x (1-x)^2 \geq 0.
            \end{split}
        \end{equation}
    \qed
\end{proof}


\section*{Acknowledgements}
The authors would like to thank Mohammad Hassan Ahmad Yarandi and Luca Ganassali for their helpful comments on earlier versions of the article.

\fund 
There are no funding bodies to thank relating to this creation of this article.

\competing 
There were no competing interests to declare which arose during the preparation or publication process of this article.


%
%
%

\end{document}